\documentclass[english]{amsart}
\usepackage[T1]{fontenc}
\usepackage[latin9]{inputenc}
\usepackage{array}
\usepackage{booktabs}
\usepackage{multirow}
\usepackage{amstext}
\usepackage{amsthm}
\usepackage{amssymb}
\usepackage{graphicx}

\makeatletter

\providecommand{\tabularnewline}{\\}

\numberwithin{equation}{section}
\numberwithin{figure}{section}
\theoremstyle{plain}
\newtheorem*{thm*}{\protect\theoremname}
\theoremstyle{remark}
\newtheorem*{rem*}{\protect\remarkname}
\theoremstyle{definition}
\newtheorem*{example*}{\protect\examplename}
\theoremstyle{plain}
\newtheorem*{prop*}{\protect\propositionname}
\theoremstyle{definition}
\newtheorem*{defn*}{\protect\definitionname}
\theoremstyle{plain}
\newtheorem{thm}{\protect\theoremname}
\theoremstyle{plain}
\newtheorem{lem}[thm]{\protect\lemmaname}
\theoremstyle{plain}
\newtheorem{prop}[thm]{\protect\propositionname}
\theoremstyle{plain}
\newtheorem*{cor*}{\protect\corollaryname}
\theoremstyle{definition}
\newtheorem{defn}[thm]{\protect\definitionname}
\theoremstyle{definition}
\newtheorem{example}[thm]{\protect\examplename}
\theoremstyle{plain}
\newtheorem*{lem*}{\protect\lemmaname}

\usepackage{color,graphics}
\usepackage[all]{xy}

\AtBeginDocument{
  
}

\makeatother

\usepackage{babel}
\providecommand{\corollaryname}{Corollary}
\providecommand{\definitionname}{Definition}
\providecommand{\examplename}{Example}
\providecommand{\lemmaname}{Lemma}
\providecommand{\propositionname}{Proposition}
\providecommand{\remarkname}{Remark}
\providecommand{\theoremname}{Theorem}

\begin{document}
\title{Dimension of the moduli space of a germ of curve in $\mathbb{C}^{2}$. }
\maketitle
\begin{center}
\textsc{Yohann Genzmer}\footnote{The author is partially supported by ANR-13-JS01-0002-0}
\par\end{center}

\global\long\def\D{D}%
\global\long\def\dd{\textup{d}}%
\global\long\def\DD{\overline{D}}%

\newcommand{\bigslant}[2]{{\raisebox{.2em}{$#1$}\left/\raisebox{-.2em}{$#2$}\right.}}
\begin{abstract}
In this article, we prove a formula that computes the generic dimension
of the moduli space of a germ of irreducible curve in the complex
plane. It is obtained from the study of the Saito module associated
to the curve, which is the module of germs of holomorphic $1$-forms
letting the curve invariant.
\end{abstract}

\section*{Introduction}

\global\long\def\dd{\textup{d}}%
In 1973, in its lecture \cite{zariski}, Zariski started the systematic
study of the analytic classification of \emph{the branches }of the
complex plane\emph{,} which are germs of irreducible curves at the
origin of $\mathbb{C}^{2}$. The general purpose was to describe as
accurately as possible the moduli space of $S$ that is the quotient
of the topological class of $S$ by the action of the group $\textup{Diff}\left(\mathbb{C}^{2},0\right)$, 

\begin{equation*}
\mathbb{M}\left(S\right)=\bigslant{\left\{ \left.S^{\prime}\right|S^{\prime}\textup{ topologically equivalent to } S\right\}}{\textup{Diff}\left(\mathbb{C}^{2},0\right)}
\end{equation*}

The Puiseux parametrization of a branch $S=\left\{ \left.\gamma\left(t\right)\right|t\in\left(\mathbb{C},0\right)\right\} $
written
\begin{equation}
\gamma:\begin{cases}
x=t^{p}\\
y=t^{q}+\sum_{k>q}a_{k}t^{k}
\end{cases},\ p<q,\ p\nmid q,\ t\in\left(\mathbb{C},0\right)\label{eq:1-1}
\end{equation}

highligths two basic topological invariants, namely the integers $p$
and $q.$ In the whole article, we will denote them by $p\left(S\right)$
and $q\left(S\right)$, or simply, $p$ and $q$ when no confusion
is possible. The integer $p\left(S\right)$ corresponds to the algebraic
multiplicity of the branch $S$. This is also the algebraic multiplicity
at $\left(0,0\right)$ of any irreducible function $f\in\mathbb{C}\left\{ x,y\right\} $
that vanishes along $S$. Actually, Zariski proved that the whole
topological classification depends on a sub-semigroup $\Gamma_{S}$
of $\mathbb{N}$ defined by
\[
\Gamma_{S}=\left\{ \left.\nu\left(f\circ\gamma\right)\right|f\in\mathbb{C}\left\{ x,y\right\} ,\ f\left(0\right)=0\right\} 
\]
where $\nu$ is the standard valuation of $\mathbb{C}\left\{ t\right\} $. 

Beyond the topological classification, Zariski proposed in \cite{zariski}
various approaches to achieve the analytical classification, introducing
in particular the set $\Lambda_{S}$ of valuations of Khler differential
forms for $S$ 
\[
\Lambda_{S}=\left\{ \left.\nu\left(\gamma^{*}\omega\right)+1\right|\omega\in\Omega^{1}\left(\mathbb{C}^{2},0\right)\right\} \supset\Gamma_{S}\setminus\left\{ 0\right\} 
\]

Fixing the topological type - and thus the semigroup $\Gamma_{S}$
above -, Zariski gave a precise description of the associated moduli
space for, for instance, 
\[
\Gamma_{S}=\left\langle 2,3\right\rangle ,\ \left\langle 4,5\right\rangle ,\ \left\langle 4,6,\beta_{2}\right\rangle 
\]
 or more generally $\left\langle n,n+1\right\rangle $ and $\left\langle n,hn+1\right\rangle .$
According to him, is of special interest, \emph{the generic component}
of the moduli space: a finite determinacy property ensures that $\gamma$
is analytically equivalent to a parametrization whose Taylor expansion
is truncated at an order depending on the sole topological class.
Having so a finite dimension family of branches, the theory of geometric
invariant provides an open set of orbits of same dimension under the
action of $\textup{Diff}\left(\mathbb{C},0\right)\times\textup{Diff}\left(\mathbb{C}^{2},0\right)$
- see \cite{zariski} chapter VI or \cite{MR0176983}. The image of
this open set in the moduli space is the generic component studied
by Zariski. In some sense, its dimension is the minimal number of
parameters on which a universal family for the deformation of $S$
depends. In the particular cases mentionned above, Zariski found an
explicit formula of this dimension.

In fact, as far as we know, the first example of computation of the
dimension of the generic component of the moduli space of a branch
goes back to Ebey \cite{MR0176983} who, anticipating in 1965 some
ideas of Zariski, described not only the generic component, but the
whole moduli space of the branch whose semigroup is $\left\langle 5,9\right\rangle .$
In 1978, Delorme \cite{Delorme1978} studied extensively the case
of one Puiseux pair - $\Gamma_{S}=\left\langle m,n\right\rangle $
with $m\wedge n=1$ - and established some formulas to compute the
generic dimension. In 1979, Granger \cite{Granger} and later, in
1988, Brianon, Granger and Maisonobe \cite{MR922433} produced an
algorithm to compute the generic dimension of the moduli space of
a non irreducible quasi-homogeneous curve defined by $x^{m}+y^{n}=0$
first, for $m$ and $n$ relatively prime, and then in the general
case. The common denominator of the two previous works is the algorithmic
approach based upon arithmetic properties of the continuous fraction
expansion associated to the pair $\left(m,n\right).$ In 1988, Laudal,
Martin and Pfister in \cite{MR1101844}, improved the work of Delorme
and gave an explicit description of a universal family for $S$ with
$\Gamma_{S}=\left\langle m,n\right\rangle $, $m\wedge n=1$ and a
stratification of the moduli space. Finally, in 1998, Peraire exhibited
an algorithm in \cite{Peraire} to compute the Tijuna number for a
curve in its generic component when $\Gamma_{S}=\left\langle m,n\right\rangle $,
$m\wedge n=1$, which is linked to the dimension of the generic component. 

From 2009, in a series of papers \cite{MR2509045,MR2781209,MR2996882},
Hefez and Hernandes achieved a impressive breakthrough in the problem
of Zariski. They completed the analytical classification of irreducible
germs of curves thanks to the set of valutations of Khler differential
forms. Moreover, they built an algorithm that describes very precisely
the stratification of the moduli space in terms of the possible $\Lambda_{S}$
for a given topological class, computes the dimension of each stratum
and produces some normal forms corresponding to each stratum. One
could consider that these works gave a definitive answer to the initial
problem adressed by Zariski. Nevertheless, the disadvantage of the
algorithmic approach is twofold: first, the high complexity of the
algorithm - based upon Groebner basis routine - prevents its actual
effectiveness as soon as the degree of the curve is big. Second, it
is difficult to extract general geometric informations or formulas
from it. 

In 2010 and 2011, in \cite{MR2808211,PaulGen}, Paul and the author
described the moduli space of a topologically quasi-homogeneous curve
$S$ as the spaces of leaves of an algebraic foliation defined on
the moduli of a foliation whose analytic invariant curve is precisely
$S.$ These works initiated an approach based upon the theory of foliations,
which is at stake here. 

In this article, we propose a construction relying basically, on one
hand, on the desingularization of the curve $S$, on the other hand,
on technics from the framework of the theory of holomorphic foliations.
We intend to obtain an explicit formula for the \emph{generic} dimension
of the moduli space - the dimension of the generic stratum - , that
can be performed \emph{by hand. }

\subsection*{The dimension of the generic stratum.}

Let $S$ be a germ of irreducible curve in the complex plane. 
\begin{thm*}
[Dimension]\label{main}Let $E=E_{1}\circ\cdots\circ E_{N}$ be the
minimal desingularization of $S.$ Let $c_{i}$ be the center of $E_{i}.$
Then 
\[
\dim_{\textup{gen}}\mathbb{M}\left(S\right)=\sum_{i=1}^{N}\sigma\left(\nu_{c_{i}}\left(\left(E_{1}\circ\cdots\circ E_{i-1}\right)^{-1}\left(S\right)\right)\right)
\]
where $\nu_{\star}$ is the algebraic multiplicity at $\star$ and
$\sigma\left(k\right)=\begin{cases}
\frac{\left(k-3\right)^{2}}{4} & \textup{ if }k\textup{ is odd}\\
\frac{\left(k-2\right)\left(k-4\right)}{4} & \textup{ else}
\end{cases}.$
\end{thm*}
Notice that this formula depends only on some topological invariants
of the curve $S$: in particular, it is not necessary to exhibit a
curve in the generic component of the moduli space of $S$ - that
is in general difficult - to perform the computation above. One can
take any curve in the topological class of $S$ to compute the multiplicities
involved in Theorem \ref{main}. 
\begin{rem*}
Actually, the proof performed here will lead us to a slightly more
general result where the formula keeps on being the same but appears
to be correct for any germ of curve of the form 
\[
S\cup d
\]
where $d$ will be called a \emph{direction} for $S$ and will be
defined later in the article. This trick will be helpful for the whole
induction structure of the proof. However, for the sake of simplicity,
we do not mention it directly in the theorem. 
\end{rem*}
\begin{example*}
In \cite{zariski}, Zariski showed that the dimension of the generic
component of the moduli space of $S=\left\{ y^{n}-x^{n+1}=0\right\} $
is $\sigma\left(n\right).$ After one blowing-up $E_{1}$, the strict
transform of $S$ by $E_{1}$ is a smooth curve tangent to the exceptional
divisor, thus for any $i\geq2$, the multiplicity satisfy 
\[
\nu_{c_{i}}\left(\left(E_{1}\circ\cdots\circ E_{i-1}\right)^{-1}\left(S\right)\right)\leq3.
\]
\end{example*}
\begin{example*}
More generally, for the semi-group $\Gamma_{S}=\left\langle n,nh+1\right\rangle $
with $h\geq1$, the desingularization of $S$ consists first in $h$
successive blowing-ups, after which the curve is smooth. The algebraic
multiplicity of the curve $S$ is $n.$ After $k\leq h$ blowing-ups,
the strict transform of $S$ is a curve whose topological class is
given by the semi-group $\left\langle n,n\left(h-k\right)+1\right\rangle $
that is transverse to the exceptional divisor. Thus, according to
Theorem \ref{main}, one has

\begin{align*}
\dim_{\textup{gen}}\mathbb{M}\left(S_{\left\langle n,nh+1\right\rangle }\right) & =\sigma\left(n\right)+\underbrace{\sigma\left(n+1\right)+\cdots+\sigma\left(n+1\right)}_{h-1}+\sigma\left(3\right)+\cdots\\
 & =\sigma\left(n\right)+\left(h-1\right)\sigma\left(n+1\right).
\end{align*}
This formula coincides with the one in \cite{zariski}. 
\end{example*}
\begin{example*}
Let us consider the following Puiseux parametrization 
\[
S:\begin{cases}
x & =t^{8}\\
y & =t^{20}+t^{30}+t^{35}
\end{cases}.
\]

Its semigroup is $\left\langle 8,20,50,105\right\rangle $ and its
Puiseux pairs are $(2,5)$, $(2,15)$ and $(2,35)$. Thus, $S$ is
not topologically quasi-homogeneous. The successive multiplicities
$\nu_{c_{i}}\left(\left(E_{1}\circ\cdots\circ E_{i-1}\right)^{-1}\left(S\right)\right)$
are 
\[
8,\ 9,\ 5,\ 6,\ 5,\ 5,\ 3,~..
\]
Thus the generic dimension of the moduli space is 
\[
\sigma\left(8\right)+\sigma\left(9\right)+\sigma\left(5\right)+\sigma\left(6\right)+\sigma\left(5\right)+\sigma\left(5\right)=20
\]
which is confirmed by the algorithm of Hefez and Hernandes.
\end{example*}

\subsection*{The Saito module of a germ of curves in $\left(\mathbb{C}^{2},0\right)$}

The inductive form of the formula in the main theorem comes naturally
from the inductive structure of the desingularization. At each step,
the theory of foliations is involved through the theory of logarithmic
vector fields or forms introduced by Saito in 1980 in \cite{MR586450}.
Let us consider the set $\Omega^{1}\left(S\right)$ of germs of holomorphic
one forms $\omega$ that let invariant $S,$ $\gamma^{*}\omega=0.$
Saito proved that $\Omega^{1}\left(S\right)$ is a free $\mathcal{O}_{2}-$module
of rank $2$. If $f$ is a reduced equation of $S,$ then $\frac{\omega}{f}$
is logarithmic in the original sense of Saito - see \cite{MR704017},
chapter II. Adapting the criterion of Saito for the existence of a
basis, the family $\left\{ \omega_{1},\omega_{2}\right\} $ is a basis
of $\Omega^{1}\left(S\right)$ if and only if there exists a germ
of unity $u\in\mathcal{O}$, $u\left(0\right)\neq0$ such that the
exterior product of $\omega_{1}$ and $\omega_{2}$ is written 
\[
\omega_{1}\wedge\omega_{2}=uf\dd x\wedge\dd y.
\]
In other words, the tangency locus between $\omega_{1}$ and $\omega_{2}$
is reduced to the sole curve $S$. Beyond this characterization, very
few is known about these two generators. At first glance, we can say
the following: among all the possible basis $\left\{ \omega_{1},\omega_{2}\right\} $,
there is one for which the sum of the algebraic multiplicities 
\begin{equation}
\nu\left(\omega_{1}\right)+\nu\left(\omega_{2}\right)\label{eq:190}
\end{equation}
 is maximal. According to the Saito criterion, 
\[
\nu\left(\omega_{1}\right)+\nu\left(\omega_{2}\right)\leq\nu\left(\omega_{1}\wedge\omega_{2}\right)\leq\nu\left(f\right)=\nu\left(S\right).
\]

thus the sum (\ref{eq:190}) cannot exceed $\nu\left(S\right).$ It
can be seen that
\begin{prop*}
The couple of multiplicities $\left(\nu\left(\omega_{1}\right),\nu\left(\omega_{2}\right)\right)$,
up to order, that maximizes its sum is an analytic invariant of $S$. 
\end{prop*}
However, these two integers as well as their sum are not topologically
invariant and in the topological class of a curve, they may vary widely. 
\begin{example*}
Let $S$ be the curve $y^{p}-x^{q}=0$. Then the family 
\[
\left\{ px\dd y-qy\dd x,\dd\left(y^{p}-x^{q}\right)\right\} 
\]
is a basis of the Saito module since 
\[
\left(px\dd y-qy\dd x\right)\wedge\dd\left(y^{p}-x^{q}\right)=-pq\left(y^{p}-x^{q}\right)\dd x\wedge\dd y.
\]
In that case, the couple of valuation is $\left(1,p-1\right)$ whose
sum is exactly $p.$ 
\end{example*}
\begin{example*}
However, perturbing a bit $S$, when for instance $p=6$ and $q=7$
leads to different values of the multiplicities. For instance, if
$S$ if the curve $y^{6}-x^{7}+x^{4}y^{4}=0$ which is topologically
but not analytically equivalent to $y^{6}=x^{7},$ one can show that
the couple 
\begin{eqnarray*}
\omega_{1} & = & \frac{5}{3}x^{4}\dd x-\frac{20}{21}x^{2}y^{3}\dd y+\left(\frac{8}{21}xy^{3}+y\right)\left(6x\dd y-7y\dd x\right)\\
\omega_{2} & = & \frac{20}{21}x^{3}y^{3}\dd x+\left(\frac{10}{7}y^{4}-\frac{80}{147}xy^{6}\right)\dd y+\left(x^{2}+\frac{32}{147}y^{6}\right)\left(6x\dd y-7y\dd x\right)
\end{eqnarray*}
is a basis for $\Omega^{1}\left(S\right)$. The multiplicities are
respectively $2$ and $3$ whose sum is strictly smaller than the
multiplicity of $S.$ 
\end{example*}
\begin{example*}
Finally, if $S$ is given by $y^{6}-x^{7}+y^{2}x^{5}=0$, an other
perturbation of $y^{6}-x^{7}=0$, then it can be seen that $S$ admits
a basis $\left\{ \omega_{1},\omega_{2}\right\} $ with $\nu\left(\omega_{1}\right)=\nu\left(\omega_{2}\right)=3.$ 
\end{example*}
This example leads us to introduce the following class of curves.
\begin{defn*}
A curve $S$, reducible or not, is said to admit \emph{a balanced
basis} if there exists a basis $\left\{ \omega_{1},\omega_{2}\right\} $
of $\Omega^{1}\left(S\right)$ with 

\begin{itemize}
\item $\nu\left(\omega_{1}\right)=\nu\left(\omega_{2}\right)=\frac{\nu\left(S\right)}{2}$
if $\nu\left(S\right)$ is even,
\item $\nu\left(\omega_{1}\right)=\nu\left(\omega_{2}\right)-1=\frac{\nu\left(S\right)-1}{2}$
else. 
\end{itemize}
\end{defn*}
A \emph{direction} $d$ for $S$ is either an empty set, a smooth
germ of curve or the union of two transverse smooth curves. The interest
of $d$ will be highlighted in the course of the article. We will
denote by $S_{d}$ the union $S\cup d$. The following result will
be the key to prove the formula in the main theorem
\begin{thm}
\label{conj2}For a \emph{generic} irreducible curve $S$ and any
direction $d$, one has 
\[
\min_{\omega\in\Omega^{1}\left(S_{d}\right)}\nu\left(\omega\right)=\left[\frac{\nu\left(S_{d}\right)}{2}\right]
\]
where $\left[\cdot\right]$ stands for the integer part function.
Moreover, if $S$ is generic, for any direction $d,$ the curve $S_{d}$
admits a balanced basis.
\end{thm}

This result will be a consequence of a construction of a very particular
element in the Saito module of $S_{d}$. This construction will be
based upon an arithmetic property of the reduction of singularities
following some results of Wall \cite{MR2107253} and a recipe to produce
foliations with desired invariant curves inspired by \cite{MR2422017,alcides}.
\begin{thm}
\label{thm:If--is}If $\nu\left(S_{d}\right)$ is even or if $d$
is empty or reduced to one component, then there exists a $1-$form
$\omega$ of multiplicity $\left[\frac{\nu\left(S_{d}\right)}{2}\right]$
in $\Omega^{1}\left(S_{d}\right)$ whose induced foliation is not
\emph{dicritical} along the exceptional divisor of the standard blowing-up
of its singularity, which means that the strict transform of $\omega$
by $E_{1}$ let invariant $E_{1}^{-1}\left(0\right).$
\end{thm}

\subsection*{Structure of the article. }

The structure of the proof of the main theorem is 
\begin{quote}
Theorem \ref{thm:If--is} $\Longrightarrow$ Theorem \ref{conj2}
$\Longrightarrow$ Main Theorem
\end{quote}
The first section of this article is devoted to the proof of the second
implication. The second focuses on the proof of Theorem \ref{thm:If--is}.
Finally the last contains the proof of the first implication. 

\section{Dimension of the moduli space \& Theorem \ref{conj2} $\protect\Longrightarrow$
Main Theorem \label{sec:Generic-dimension-of}}

To describe the contribution of the deformation theory, let us introduce
first some notations that will be used all along the article. 

Let $E$ be the minimal log-canonical resolution of $S$. We denote
it by 
\[
E:\left(\mathcal{M},D\right)\to\left(\mathbb{C}^{2},0\right).
\]
The map $E$ is a finite sequence of elementary blowing-ups of points
\[
E=E_{1}\circ E_{2}\circ\cdots\circ E_{N}.
\]
If $\Sigma$ is a germ of curve at $\left(\mathbb{C}^{2},0\right)$
or a divisor, $\Sigma^{E}$ will stand for the strict transform of
$\Sigma$ by $E$, i.e., the closure in $\mathcal{M}$ of $E^{-1}\left(\Sigma\setminus\left\{ 0\right\} \right).$ 

The exceptional divisor of $E$, $D=E^{-1}\left(0\right)$, is an
union of a finite number of exceptional smooth rational curves intersecting
transversely 
\[
D=\bigcup_{i=1}^{N}D_{i},\qquad D_{i}\simeq\mathbb{P}^{1}\left(\mathbb{C}\right).
\]
The components are numbered such that $D_{i}$ appears exactly after
$i$ blowing-ups. Finally, let us denote $E^{j}$ the truncated process
\[
E^{j}=E_{j}\circ E_{2}\circ\cdots\circ E_{N}\textup{ and }D^{j}=\bigcup_{i=j}^{N}D_{i}.
\]

The initial lemma is the following
\begin{lem}
\label{thm:Let--be2-1}Let $TS$ be the a sheaf of base $D$ whose
stalk at a point $x\in D$ is the set of germs of tangent vector fields
to the total transform of $S$ by $E.$ Then the generic dimension
of the moduli space of $S$ is 
\[
\dim_{\mathbb{C}}H^{1}\left(D,TS_{\textup{gen}}\right)
\]
where $S_{\textup{gen}}$ is a curve in the generic component of the
moduli space of $S.$
\end{lem}

\begin{proof}
In \cite{zariski}, Zariski proved that the dimension of the generic
component is equal to the dimension of the space of parameters of
a semi-universal deformation of any curve $S_{\textup{gen}}$ in the
generic component of the moduli space of $S.$ On the other hand,
J.-F. Mattei proved in \cite{MatQuasi} that any curve $S$ admits
a semi-universal deformation whose base space is $\left(\mathbb{C}^{\dim_{\mathbb{C}}H^{1}\left(D,TS_{\textup{gen}}\right)},0\right),$
which conclude the proof. 
\end{proof}
Let $S$ be a curve - irreducible or not -, $E_{1}$ be the standard
blow-up and $D_{1}=E_{1}^{-1}\left(0\right)$
\begin{prop}
\label{dimension} If the module of Saito $\Omega^{1}\left(S\right)$
admits a basis $\left\{ \omega_{1},\omega_{2}\right\} $ with 
\[
\nu\left(\omega_{1}\right)+\nu\left(\omega_{2}\right)=\nu\left(S\right)
\]
 Then 
\[
\dim_{\mathbb{C}}H^{1}\left(D_{1},TS\right)=\frac{\left(\nu_{1}-1\right)\left(\nu_{1}-2\right)}{2}+\frac{\left(\nu_{2}-1\right)\left(\nu_{2}-2\right)}{2}
\]
with $\nu_{i}=\nu\left(\omega_{i}\right)$.
\end{prop}

\begin{proof}
Since $\left\{ \omega_{1},\omega_{2}\right\} $ is a basis of $\Omega^{1}\left(S\right)$,
the criterion of Saito ensures that 
\[
\omega_{1}\wedge\omega_{2}=uf\dd x\wedge\dd y.
\]
for some unity $u$ and some reduced equation $f$ of $S$. Let $X_{1}$
and $X_{2}$ be the two vector fields defined by 
\[
X_{i}=\omega_{i}^{\sharp}=i_{X_{i}}\left(\dd x\wedge\dd y\right)
\]
 where $i_{\star}$ is the inner product. One can write 
\begin{equation}
\det\left(X_{1},X_{2}\right)=uf.\label{eq:12}
\end{equation}
Let us consider the standard covering of $D_{1}$ by two open sets
$U_{1}$ and $U_{2}$ and two charts $\left(x_{1},y_{1}\right)$ and
$\left(x_{2},y_{2}\right)$ with 
\[
y_{2}=y_{1}x_{1}\qquad x_{2}=\frac{1}{y_{1}}\qquad E_{1}\left(x_{1},y_{1}\right)=\left(x_{1},y_{1}x_{1}\right).
\]
The pull-back of (\ref{eq:12}) by $E_{1}$ is written in the first
chart
\[
\det\left(E_{1}^{*}X_{1},E_{1}^{*}X_{2}\right)=\frac{E_{1}^{*}uE_{1}^{*}f}{\det E_{1}}.
\]
Dividing by $x^{\nu}=x^{\nu_{1}+\nu_{2}}$ yields the relation 
\[
\det\left(\tilde{X}_{1}^{1},\tilde{X}_{2}^{1}\right)=E_{1}^{*}u\tilde{f}x_{1}
\]
where $\tilde{X}_{i}^{1}=\frac{E_{1}^{*}X_{i}}{x_{1}^{\nu_{i}-1}}$.
The two vector fields $\tilde{X}_{1}^{1}$ and $\tilde{X}_{2}^{1}$
are tangent to the exceptional divisor. Obviously, they are also tangent
to $\tilde{f}=0$. According to the Saito criterion, at any point
$c$ of the exceptional divisor, the germ of $\left\{ \tilde{X}_{1}^{1},\tilde{X}_{2}^{1}\right\} $
at $c$ is a basis of the module $\left(TS\right)_{c}$. The computation
works the same in the second chart $\left(x_{2},y_{2}\right)$ of
the blow-up. 

The open sets $U_{1}$ and $U_{2}$ are Stein. Thus following \cite{SiuThm},
those admit a system of Stein neighborhoods. Since $TS$ is coherent,
by inductive limit, we deduce that the covering $\left\{ U_{1},U_{2}\right\} $
is acyclic for $TS.$ Therefore, one can compute the cohomology using
this covering and thus 
\[
H^{1}\left(D_{1},TS\right)=H^{1}\left(\left\{ U_{1},U_{2}\right\} ,TS\right)=\frac{H^{0}\left(U_{1}\cap U_{2},TS\right)}{H^{0}\left(U_{1},TS\right)\oplus H^{0}\left(U_{2},TS\right)}.
\]
Now, the spaces of global sections on $U_{1}$, $U_{2}$ and the intersection
can be described as follows 
\begin{eqnarray*}
H^{0}\left(U_{1}\cap U_{2},TS\right) & = & \left\{ \left.\phi_{12}\tilde{X}_{1}^{1}+\psi_{12}\tilde{X}_{2}^{1}\right|\phi_{12},\psi_{12}\in\mathcal{O}\left(U_{1}\cap U_{2}\right)\right\} \\
H^{0}\left(U_{1},TS\right) & = & \left\{ \left.\phi_{1}\tilde{X}_{1}^{1}+\psi_{1}\tilde{X}_{2}^{1}\right|\phi_{1},\psi_{1}\in\mathcal{O}\left(U_{1}\right)\right\} \\
H^{0}\left(U_{2},TS\right) & = & \left\{ \left.\phi_{2}\tilde{X}_{1}^{2}+\psi_{2}\tilde{X}_{2}^{2}\right|\phi_{2},\psi_{2}\in\mathcal{O}\left(U_{2}\right)\right\} .
\end{eqnarray*}
Thus, the cohomological equation is written 
\begin{eqnarray*}
\phi_{12}\tilde{X}_{1}^{1}+\psi_{12}\tilde{X}_{2}^{1} & = & \phi_{1}\tilde{X}_{1}^{1}+\psi_{1}\tilde{X}_{2}^{1}-\phi_{2}\tilde{X}_{1}^{2}+\psi_{2}\tilde{X}_{2}^{2}\\
 & = & \phi_{1}\tilde{X}_{1}^{1}+\psi_{1}\tilde{X}_{2}^{1}-\phi_{2}y_{1}^{-\nu_{1}+1}\tilde{X}_{1}^{1}+\psi_{2}y_{1}^{-\nu_{2}+1}\tilde{X}_{2}^{1}.
\end{eqnarray*}
Since, $\left\{ \tilde{X}_{1}^{1},\tilde{X}_{2}^{1}\right\} $ is
a basis of $\mathcal{O}$-module, the above leads to the system 
\[
\begin{cases}
\phi_{12} & =\phi_{1}-\phi_{2}y_{1}^{-\nu_{1}+1}\\
\psi_{12} & =\psi_{1}-\psi_{2}y_{1}^{-\nu_{2}+1}
\end{cases}.
\]
Writing these equations using Taylor expansions leads to the checked
number of obstructions.
\end{proof}
Finally, the proof of 
\begin{quote}
Theorem \ref{conj2}$\Longrightarrow$ Main Theorem.
\end{quote}
goes as follows. Consider the covering $\left\{ U,V\right\} $ of
$D_{1}$ where $V$ is a very small ball around the singular point
of $S_{\textup{gen}}^{E_{1}}$ and $U=D_{1}\setminus\textup{Sing\ensuremath{\left(S_{\textup{gen}}^{E_{1}}\right)}}$ 

\begin{figure}
\begin{centering}
\includegraphics[scale=0.7]{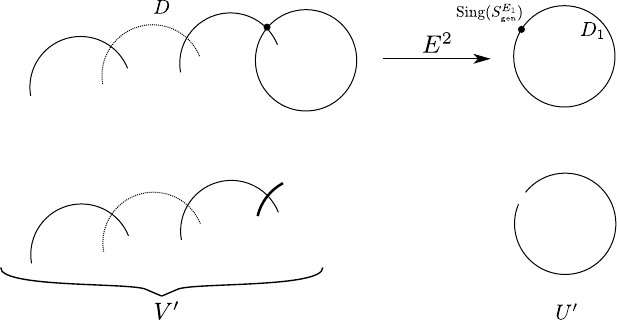}
\par\end{centering}
\caption{\label{fig:Covering-of-}Covering of $D$ adapted to the Mayer-Vietoris
argument.}

\end{figure}

The set 
\[
\left\{ U^{\prime}=\left(E^{2}\right)^{-1}\left(U\right),V^{\prime}=\left(E^{2}\right)^{-1}\left(V\right)\right\} 
\]
consists in a covering of $D$ and $V^{\prime}$ is a neighborhood
of $D_{2}$ as shown in Figure (\ref{fig:Covering-of-}). The Mayer-Vietoris
sequence associated to this covering and applied to the sheaf $TS_{\textup{gen}}$
leads to the following long exact sequences in cohomology 
\[
0\to N\to H^{1}\left(D,TS_{\textup{gen}}\right)\to H^{1}\left(V^{\prime},TS_{\textup{gen}}\right)\oplus H^{1}\left(U^{\prime},TS_{\textup{gen}}\right)\to H^{1}\left(V^{\prime}\cap U^{\prime},TS_{\textup{gen}}\right)
\]
where $N$ is given by the exact sequence 
\[
H^{0}\left(V^{\prime},TS_{\textup{gen}}\right)\oplus H^{0}\left(U^{\prime},TS_{\textup{gen}}\right)\to H^{0}\left(V^{\prime}\cap U^{\prime},TS_{\textup{gen}}\right)\to N.
\]
Since $V^{\prime}\cap U^{\prime}$ and $U^{\prime}$ are Stein, one
has 
\begin{align*}
H^{1}\left(V^{\prime}\cap U^{\prime},TS_{\textup{gen}}\right) & =0\\
H^{1}\left(U^{\prime},TS_{\textup{gen}}\right) & =0
\end{align*}
By inductive limit on the neighborhood of $\textup{Sing\ensuremath{\left(S_{\textup{gen}}^{E_{1}}\right)}}$,
one can show that 
\[
H^{1}\left(V^{\prime},TS_{\textup{gen}}\right)\simeq H^{1}\left(D^{2},TS_{\textup{gen}}\right)
\]
Moreover, $E^{2}$ induces the following isomorphisms
\begin{align*}
\left(E^{2}\right)^{*} & :H^{0}\left(U^{\prime},TS_{\textup{gen}}\right)\to H^{0}\left(U,TS_{\textup{gen}}\right)\\
\left(E^{2}\right)^{*} & :H^{0}\left(V^{\prime}\cap U^{\prime},TS_{\textup{gen}}\right)\to H^{0}\left(U\cap V,TS_{\textup{gen}}\right),\\
\left(E^{2}\right)^{*} & :H^{0}\left(V^{\prime},TS_{\textup{gen}}\right)\to H^{0}\left(V,TS_{\textup{gen}}\right)
\end{align*}
In the two first cases, $E^{2}$ is an isomorphism itself on involved
neightborhoods. In the third case, this is a consequence of Hartogs
extension lemma noticing that $E^{2}$ is an isomorphism from a neighborhood
of $\left(E^{2}\right)^{-1}\left(V\setminus\textup{Sing\ensuremath{\left(S_{\textup{gen}}^{E_{1}}\right)}}\right)$
to its image. The Mayer-Vietoris sequence finally decomposes $H^{1}\left(D,TS_{\textup{gen}}\right)$
along the desingularization of $S_{\textup{gen}}$: 
\[
H^{1}\left(D,TS_{\textup{gen}}\right)\simeq H^{1}\left(D_{1},TS_{\textup{gen}}\right)\bigoplus H^{1}\left(D^{2},TS_{\textup{gen}}\right).
\]
The curve $S_{\textup{gen}}$ admits a balanced basis according to
Theorem \ref{conj2}. Hence, the main theorem is an inductive\footnote{This is were the use of a direction $d$ following $S$ is useful.
Indeed, the total transform $E_{1}^{-1}\left(S_{\textup{gen}}\right)$
is not an irreducible germ of curve but the union of an irreducible
germ and of a direction which is the local trace of $D_{1}.$ } application of Proposition \ref{dimension} noticing that in that
case 
\[
\dim H^{1}\left(D_{1},TS_{\textup{gen}}\right)=\sigma\left(\nu\left(S_{\textup{gen}}\right)\right).
\]
As a corollary, the formula gives a straightforward proof of the following
result contained in \cite{MR2509045}.
\begin{cor*}
A germ of irreducible curve $S$ is generically rigid if and only
if 

\begin{itemize}
\item $\nu\left(S\right)\in\left\{ 1,2,3\right\} $ or 
\item $\nu\left(S\right)=4$ and its Puiseux pairs are $\left(4,5\right),\ \left(4,7\right)$
or $\left(\left(2,3\right),\left(2,2k+1\right)\right)$ with $k\geq3.$ 
\end{itemize}
\end{cor*}
Indeed, one can check that the cases above are the only one for which
the formula in Theorem \ref{main} yields $0.$ 

\section{A remarkable element in $\Omega^{1}\left(S\right)$ \& proof of theorem
\ref{thm:If--is}\label{sec:An-auxilliary-foliation.}}

For any basis $\left\{ \omega_{1},\omega_{2}\right\} $ of $\Omega^{1}\left(S_{d}\right)$,
the criterion of Saito ensures that 
\[
\nu\left(\omega_{1}\right)+\mbox{\ensuremath{\nu\left(\omega_{2}\right)}}\leq\nu\left(S_{d}\right).
\]
Thus at least one of these multiplicities is smaller or equal to $\left[\frac{\nu\left(S_{d}\right)}{2}\right]$,
which proves one part of the equality in Theorem \ref{conj2}. However,
to obtain the whole equality we will need some more informations about
these generators. In this section, we are going to construct quite
explicitly an element of $\Omega^{1}\left(S_{d}\right)$ with multiplicity
$\left[\frac{\nu\left(S_{d}\right)}{2}\right]$.

We recall that a foliation $\mathcal{F}$ is said to be \emph{dicritical
along a divisor $\Sigma$} if and only if $\mathcal{F}$ is generically
transverse to $\Sigma.$ 

Let us give a sketch of the proof of Theorem \ref{thm:If--is}. First,
we construct an auxiliary foliation $\mathcal{F}\left[S_{d}\right]$
tangent to some curve $\mathfrak{S}$ topologically equivalent to
$S_{d}$ - but not necessarly analytically equivalent to $S_{d}$
- with the desired algebraic multiplicity. Then, we study the deformations
of $\mathcal{F}\left[S_{d}\right]$ by means of cohomological tools.
In particular, considering a deformation linking $\mathfrak{S}$ to
$S_{d,}$ we prove that it can be followed by a deformation of $\mathcal{F}\left[S_{d}\right]$
that preserves the algebraic multiplicity. The resulting foliation
is tangent to $S_{d}$ with $\left[\frac{\nu\left(S_{d}\right)}{2}\right]$
as algebraic multiplicity. Among other properties, we obtain Theorem
\ref{thm:If--is}.

\subsection{The auxiliary foliation $\mathcal{F}\left[S_{d}\right]$.}

In this section, we are going to construct a foliation associated
to $S_{d}$, denoted by $\mathcal{F}\left[S_{d}\right]$, thanks to
a result of Alcides Lins-Neto \cite{alcides,Loray} that is a kind
of \emph{recipe }to construct germs of singular foliations in the
complex plane. 

Let $E$ be the minimal desingularization of $S$. We denote it by
\[
E:\left(\mathcal{M},D\right)\to\left(\mathbb{C}^{2},0\right).
\]
Recall that $E$ is a finite sequence of elementary blowing-ups of
points
\[
E=E_{1}\circ E_{2}\circ\cdots\circ E_{N}.
\]
We can encode the map $E$ in a square matrix $\mathcal{E}$ of size
$N$ called by Wall the \emph{proximity matrix} \cite[p. 52]{MR2107253}.
The first two columns of $\mathcal{E}$ are $\underbrace{\left(\begin{array}{ccc}
1 & -1\\
0 & 1 & \ldots\\
0 & 0\\
\vdots & \vdots\\
0 & 0
\end{array}\right)}_{N}$. The $i$th column $C_{i}$ is defined by $\left(C_{i}\right)_{i}=1$
and $\left(C_{i}\right)_{i-1}=-1$ ; if $E_{i}$ is the blowing-up
of the point $D_{i-1}\cap D_{j}$ then $\left(C_{i}\right)_{j}=-1$
; for any other index $j$, $\left(C_{i}\right)_{j}=0.$ Notice that,
since the curve S is irreducible, the proximity matrix has the following
property: if $i<j$ and $C_{ij}=0$ then $C_{ik}=0$ for $k\geq j$. 

Let $S_{i}$ be the strict transform of $S$ by $E_{1}\circ\cdots\circ E_{i-1}$
for $i\geq2$ and $S_{1}=S.$ The map $E^{i}$ is the minimal desingularization
of the total transform of $S_{1}$ by $E_{1}\circ\cdots\circ E_{i-1}.$ 
\begin{defn}
Let $E:\left(\mathcal{M},D\right)\to\left(\mathbb{C}^{2},0\right)$
be a process of blow-ups $E=E_{1}\circ\cdots\circ E_{p}$. Let us
write $D=E^{-1}\left(0\right)=\bigcup_{i=1}^{p}D_{i}.$ Let $\mathfrak{M}$
be the maximal ideal at $\left(\mathbb{C}^{2},0\right)$ and $\mathcal{I}$
the sheaf over $D$ of ideals generated locally by the functions of
the form $g\circ E$ where $g\in\mathfrak{M}.$ Then $\mathcal{I}$
can be decomposed the following way 
\[
\mathcal{I}=\prod_{i=1}^{p}\mathcal{I}_{D}^{n\left(E,D\right)}
\]
where $\mathcal{I}_{D}$ is the sheaf of functions vanishing on $D$
and $n\left(E,D\right)$ are some integers depending on $E$ and $D.$
The integer $n\left(E,D\right)$ is called the \emph{multiplicity}
of $D$ with respect to $E.$
\end{defn}

The following lemma is in \cite[p. 53]{MR2107253} 
\begin{lem}
The inverse of the proximity matrix $\mathcal{E}^{-1}$ has the following
form 
\[
\left(\begin{array}{cccc}
1\\
0 & \ddots & e_{kl}\\
 & \ddots & 1\\
0 &  & 0 & 1
\end{array}\right)
\]

where $e_{kl}=n\left(E^{k},D_{l}\right).$ Furthermore, the matrix
$-\mathcal{E}\left(^{t}\mathcal{E}\right)$ is the intersection matrix
of $D$. 
\end{lem}

\begin{example}
Let us consider $S=\left\{ y^{5}=x^{13}\right\} .$ Then the proximity
matrix $\mathcal{E}$ is written 
\[
\mathcal{E}=\left(\begin{array}{cccccc}
1 & -1 & 0 & 0 & 0 & 0\\
0 & 1 & -1 & -1 & 0 & 0\\
0 & 0 & 1 & -1 & -1 & 0\\
0 & 0 & 0 & 1 & -1 & -1\\
0 & 0 & 0 & 0 & 1 & -1\\
0 & 0 & 0 & 0 & 0 & 1
\end{array}\right).
\]
The inverse matrix is written 
\[
\mathcal{E}^{-1}=\left(\begin{array}{cccccc}
1 & 1 & 1 & 2 & 3 & 5\\
0 & 1 & 1 & 2 & 3 & 5\\
0 & 0 & 1 & 1 & 2 & 3\\
0 & 0 & 0 & 1 & 1 & 2\\
0 & 0 & 0 & 0 & 1 & 1\\
0 & 0 & 0 & 0 & 0 & 1
\end{array}\right).
\]
The exceptional divisors of the associated sequence of processes of
blowing-ups $\left\{ E^{k}\right\} _{k=1..5}$ are presented in Figure
(\ref{fig:Sequence-of-process}).

\begin{figure}
\begin{centering}
\includegraphics[scale=0.3]{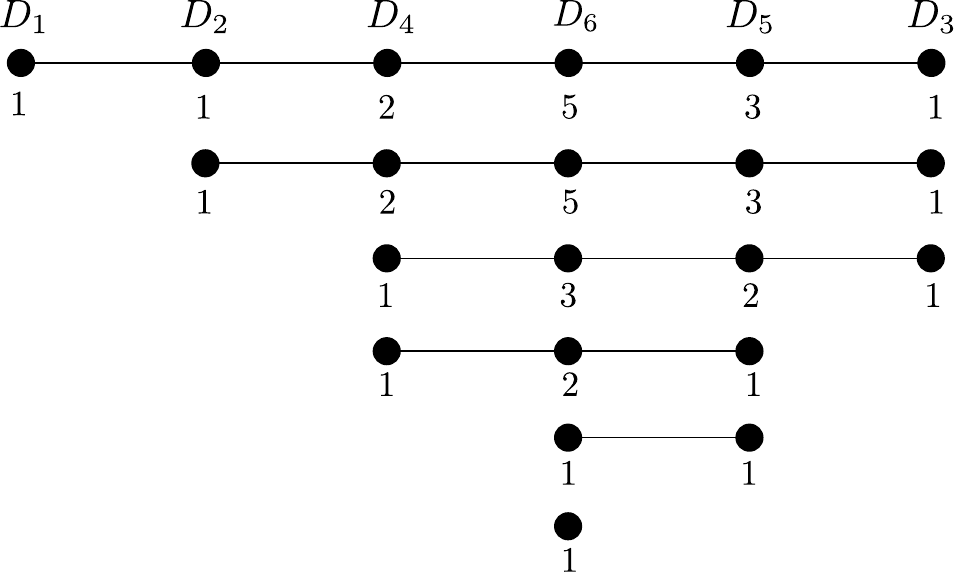}
\par\end{centering}
\caption{\label{fig:Sequence-of-process}Exceptional divisors of the sequence
of processes of blowing-ups associated to the desingularization of
$y^{5}-x^{13}=0.$}
\end{figure}

Notice that, as soon as $S$ is singular, for any direction $d$,
$S$ and $S_{d}$ share the same reduction. The next proposition is
the one upon which the construction of the auxiliary foliation $\mathcal{F}\left[S_{d}\right]$
is based.
\end{example}

\begin{prop}
\label{lem-combi}Let $\delta_{1}\in\left\{ 0,1,2\right\} $ be the
number of components of the direction $d.$ In the same way, consider
the number $\delta_{i}$ of branches of $\left(E_{1}\circ\cdots\circ E_{i-1}\right)^{-1}\left(d\right)$
meeting $S_{i}$ for $2\leq i\leq N.$ For $i\geq2,$ $\delta_{i}\in\left\{ 1,2\right\} .$
Let us denote $n_{i}-1$ the number of $-1$ on the $i$-th row of
$\mathcal{E}$. 

Let us consider the vector of integers defined by 
\begin{equation}
\left(\begin{array}{c}
p_{1}\\
p_{2}\\
\vdots\\
p_{N}
\end{array}\right)=\mathcal{E}\left(\begin{array}{c}
\left[\frac{\nu\left(S_{1}\right)-\delta_{1}}{2}\right]+1\\
\left[\frac{\nu\left(S_{2}\right)-\delta_{2}}{2}\right]+1\\
\vdots\\
\left[\frac{\nu\left(S_{N}\right)-\delta_{N}}{2}\right]+1
\end{array}\right).\label{eq:SuperRelation}
\end{equation}
Then 

\begin{enumerate}
\item any integer $p_{i}$ is bigger or equal to $-1.$ The case $p_{i}=-1$
occurs if only if
\begin{enumerate}
\item either, $n_{i}=2,$ $\delta_{i}=2,\ \delta_{i+1}=1$ and $\nu\left(S_{i}\right)=p\left(S_{i}\right)$
is odd.
\item or, $n_{i}=3,$ $\delta_{i}=2,\ \delta_{i+1}=1$, $\nu\left(S_{i}\right)$
is odd and $q\left(S_{i}\right)$ is even. 
\end{enumerate}
\item If $D_{i}\cap D_{j}\neq\emptyset$ then one cannot have both $p_{i}=-1$
and $p_{j}=-1.$ 
\item Let us consider $\overline{D}$ the exceptional divisor $D$ deprived
of $D_{N}$ and of the components $D_{i}$ for which $p_{i}=-1$.
Then in each connected component of $\overline{D}$, there exists
at least one component $D_{j}$ for which, either $p_{j}>0$ or, that
meets a component of $d^{E}.$
\item $p_{N}=0.$
\end{enumerate}
\end{prop}

\begin{proof}
The proof is an induction on the length of the desingularization of
$S$. Let us consider that $\mathcal{E}$ is written 
\[
\mathcal{E}=\left(\begin{array}{ccccccc}
1 & -1 & -1 & \cdots & -1 & 0\\
 & 1 & -1\\
 &  & 1 &  &  &  & \cdots\\
 &  &  & \ddots & -1\\
 &  &  &  & 1 & -1\\
 &  &  &  &  & 1\\
 &  & \vdots &  &  &  & \ddots
\end{array}\right)
\]
Expanding the expression of $p_{1}$, we find 
\[
p_{1}=\left[\frac{\nu\left(S_{1}\right)-\delta_{1}}{2}\right]+1-\sum_{j=2}^{n}\left(\left[\frac{\nu\left(S_{j}\right)-\delta_{j}}{2}\right]+1\right)
\]
where for the sake of simplicity $n=n_{1}$. \footnote{Actually, $n_{1}$ is equal to $\left\lceil \frac{q}{q-p}\right\rceil $,
but we will not need this expression.} Consider a Puiseux parametrization of $S_{1}=S,$ 
\[
S_{1}:\begin{cases}
x=t^{p}\\
y=t^{q}+\cdots
\end{cases}
\]
with $p=p\left(S_{1}\right)<q=q\left(S_{1}\right).$ Following to
the desingularization of $S_{1},$ encoded in the proximity matrix,
the multiplicities and the $\delta_{i}$'s satisfy 
\begin{eqnarray*}
\nu\left(S_{1}\right) & = & p\\
\nu\left(S_{j}\right) & = & q-p\ \text{\textup{ for }}2\leq j\leq n-1\\
\nu\left(S_{n}\right) & = & \left(n-1\right)p-\left(n-2\right)q\\
\delta_{1} & \in & \left\{ 0,1,2\right\} \\
\delta_{2} & \in & \left\{ 1,2\right\} \\
\delta_{j} & = & 2\ \textup{ for }3\leq j\leq n.
\end{eqnarray*}
Thus, the integer $p_{1}$ is written 
\[
p_{1}=\left[\frac{p-\delta_{1}}{2}\right]-\sum_{j=2}^{n-1}\left[\frac{q-p-\delta_{j}}{2}\right]-\left[\frac{\left(n-1\right)p-\left(n-2\right)q-\delta_{n}}{2}\right]-n+2.
\]

The following lemma is straightforward
\begin{lem}
\label{lem:If-,-then}If $n=2$, then $p_{1}$ is equal to 
\[
p_{1}=\left[\frac{p-\delta_{1}}{2}\right]-\left[\frac{p-\delta_{2}}{2}\right]=\begin{cases}
\delta_{1}=0,\ \delta_{2}=1 & \begin{cases}
0 & \textup{ if \ensuremath{p} is odd }\\
1 & \textup{ else}
\end{cases}\\
\delta_{1}=1 & \begin{cases}
\delta_{2}=1 & 0\\
\delta_{2}=2 & \begin{cases}
1 & \textup{ if \ensuremath{p} is odd }\\
0 & \textup{ else}
\end{cases}
\end{cases}\\
\delta_{1}=2 & \begin{cases}
\delta_{2}=1 & \begin{cases}
-1 & \textup{ if \ensuremath{p} is odd }\\
0 & \textup{ else}
\end{cases}\\
\delta_{2}=2 & 0
\end{cases}
\end{cases}.
\]
If $n\geq3$ then the values of $p_{1}$ are given in Table (\ref{tab:Values-of-.}).
When the value depends on $n,$ it is precised the value of $p_{1}$
if $n$ is even or odd. In particular, $p_{1}=-1$ if and only if
one of the following case occurs,
\begin{itemize}
\item $n=2,$ $\delta_{1}=2,$ $\delta_{2}=1$ and $p$ is odd.
\item $n=3,$ $\delta_{1}=2,$ $\delta_{2}=1$ and $p$ is odd and $q$
is even. 
\end{itemize}
\end{lem}

\begin{table}
\begin{centering}
\begin{tabular}{ccccc}
\cmidrule{2-5} \cmidrule{3-5} \cmidrule{4-5} \cmidrule{5-5} 
 & $\begin{array}{c}
p\textup{ and }q\\
\textup{both odd}
\end{array}$  & $\begin{array}{c}
p\textup{ and }q\\
\textup{both even}
\end{array}$  & $\begin{array}{c}
p\textup{ even}\\
q\textup{ odd}
\end{array}$  & $\begin{array}{c}
p\textup{ odd}\\
q\textup{ even}
\end{array}$\tabularnewline
\midrule 
\addlinespace[0.2cm]
\multirow{1}{*}{$\delta_{1}=0,\ \delta_{2}=1$} & $1$ & $1$ & $\frac{n-2}{2},\ \frac{n-1}{2}$ & $\frac{n-2}{2},\ \frac{n-3}{2}$\tabularnewline\addlinespace[0.2cm]
\midrule 
\addlinespace[0.2cm]
\multicolumn{1}{c}{$\delta_{1}=1,\ \delta_{2}=1$} & $1$ & $0$ & $\frac{n-4}{2},\ \frac{n-3}{2}$ & $\frac{n-2}{2},\ \frac{n-3}{2}$\tabularnewline\addlinespace[0.2cm]
\midrule 
\addlinespace[0.2cm]
$\delta_{1}=1,\ \delta_{2}=2$ & $1$ & $0$ & $\frac{n-2}{2},\ \frac{n-1}{2}$ & $\frac{n}{2},\ \frac{n-1}{2}$\tabularnewline\addlinespace[0.2cm]
\midrule 
\addlinespace[0.2cm]
\multicolumn{1}{c}{$\delta_{1}=2,\ \delta_{2}=1$} & $0$ & $0$ & $\frac{n-4}{2},\ \frac{n-3}{2}$ & $\frac{n-4}{2},\ \frac{n-5}{2}$\tabularnewline\addlinespace[0.2cm]
\midrule 
\addlinespace[0.2cm]
$\delta_{1}=2,\ \delta_{2}=2$ & $0$ & $0$ & $\frac{n-2}{2},\ \frac{n-1}{2}$ & $\frac{n-2}{2},\ \frac{n-3}{2}$\tabularnewline\addlinespace[0.2cm]
\bottomrule
\end{tabular}
\par\end{centering}
\bigskip{}

\caption{Valu\label{tab:Values-of-.}es of $p_{1}$ depending on $n$ being
odd or even.}
\end{table}

Now, we are able to study the general behavior of $p_{1}$ and to
prove Proposition \ref{lem-combi}. 

The property $\left(1\right)$ can be seen by reading inductively
Lemma \ref{lem:If-,-then}. 

The property $\left(2\right)$ is proved as follows. Suppose that
$p_{1}=-1$. According to property $\left(1\right),$ two cases may
occur
\begin{itemize}
\item if $n=2$, $\delta_{1}=2$ and $\delta_{2}=1$, then $D_{1}$ meets
$D_{2}$ in $D$. Since $\delta_{2}=1$, $p_{2}$ cannot be equal
to $-1.$ Proposition \ref{lem-combi} applied inductively to $S_{2}$
yields the proposition for $S_{2}.$
\item if $n=3,$ $\delta_{1}=2,$ $\delta_{2}=1$, $p$ is odd and $q$
is even, then $D_{1}$ meets $D_{3}$ and $\delta_{3}=2.$ Suppose
that $\delta_{4}=1$ then $S_{3}$ is neither tangent to $D_{1}$
nor to $D_{2}$. Looking at the Puiseux parametrization of $S_{3}$
yields 
\[
q-p=2p-q
\]
which is impossible since $p$ is odd. Thus $\delta_{4}=2$, and $p_{3}$
cannot be equal to $-1.$ We conclude by induction.
\end{itemize}
Let us now focus on property $\left(3\right).$ 
\begin{itemize}
\item Suppose first that $\delta_{1}=2.$ 
\begin{itemize}
\item If $p_{1}>0,$ then the connected component of $D_{1}$ in $\overline{D}$
contains $D_{1}$ as component with $p_{1}>0.$ Applying inductively
Proposition \ref{lem-combi} to $S_{2}$ with the sequence of $\delta$'s
equal to 
\[
\delta_{2},\ \delta_{3},\ \ldots
\]
yields the proposition for $S_{1}$ with the sequence of $\delta$
's equal to $\delta_{1},\ \delta_{2},\ \ldots$. 
\item If $p_{1}=0$, since at least one of the component of $d^{E}$ is
attached to $D_{1}$, the same argument as before ensures the proposition. 
\item If $p_{1}=-1,$ then two cases may occur :
\begin{itemize}
\item if $n=2$ then $\nu\left(S_{2}\right)=\nu\left(S_{1}\right)$ is odd
and $\delta_{2}=1.$ Applying inductively Proposition \ref{lem-combi}
to $S_{2}$ with the sequence of $\delta$'s equal to
\[
0,\ \delta_{3},\ \delta_{4},\ldots
\]
yields the result: indeed, one has 
\[
\left[\frac{\nu\left(S_{2}\right)-0}{2}\right]=\left[\frac{\nu\left(S_{2}\right)}{2}\right]=\left[\frac{\nu\left(S_{2}\right)-1}{2}\right]=\left[\frac{\nu\left(S_{2}\right)-\delta_{2}}{2}\right]
\]
and for $j\geq3$, since $n=2,$ one has $\delta_{j}^{'}=\delta_{j}$
where the $\delta_{j}^{'}$ would be the sequence obtained following
the desingularization of $S_{2}$ with $\delta_{2}^{'}=0.$ 
\item if $n=3,$ then $\delta_{2}=1$, $\nu\left(S_{1}\right)=p$ is odd
and $q$ is even. Moreover, since $n=3,$ one has $\delta_{3}=2$.
Following the desingularization of $S_{1},$ one has $\nu\left(S_{2}\right)=q-p$
that is odd and $\nu\left(S_{3}\right)=2p-q$ that is even. Applying
inductively Proposition \ref{lem-combi} to $S_{2}$ with the sequence
of $\delta'$s equal to
\[
0,\ 1,\ \delta_{4},\ldots
\]
yields the result: indeed, one has 
\begin{align*}
\left[\frac{\nu\left(S_{2}\right)-0}{2}\right] & =\left[\frac{\nu\left(S_{2}\right)-1}{2}\right]=\left[\frac{\nu\left(S_{2}\right)-\delta_{2}}{2}\right],\\
\left[\frac{\nu\left(S_{3}\right)-1}{2}\right] & =\left[\frac{\nu\left(S_{3}\right)-2}{2}\right]=\left[\frac{\nu\left(S_{2}\right)-\delta_{3}}{2}\right],
\end{align*}
and for $j\geq4$, since $n=3,$ one has $\delta_{j}^{'}=\delta_{j}$
where the $\delta_{j}^{'}$ would be the sequence obtained following
the desingularization of $S_{2}$ with $\delta_{2}^{'}=0$ and $\delta_{3}^{'}=1$. 
\end{itemize}
\end{itemize}
\item Suppose now that $\delta_{1}=1.$ Then according to property $\left(2\right)$,
$p_{1}\geq0.$ If $\delta_{2}=1$ then the component of $d^{E}$ meets
$D_{1}.$ So applying inductively Proposition \ref{lem-combi} to
$S_{2}$ with the sequence $\delta_{2},\ \delta_{3},\cdots$ yields
the proposition. Let us suppose that $\delta_{2}=2.$ If $p_{1}>0,$
then inductively the proposition is proved. If $p_{1}=0$ then according
to Lemma \ref{lem:If-,-then} two cases may occur
\begin{itemize}
\item if $n=2$ and $\nu\left(S_{2}\right)=\nu\left(S_{1}\right)$ is even,
then $D_{2}$ meets $D_{1}$ in $D$ and $p_{2}$ cannot be equal
to $-1$. Applying inductively Proposition to $S_{2}$ with the sequence
\[
1,\ \delta_{3},\ldots
\]
yields the result. The arguments are the same as before.
\item if $n\geq3$, then $p$ and $q$ are even and the curve $S$ cannot
be topologically quasi-homogeneous. While $\delta_{i}\neq1$, no component
$D_{j}$ with $p_{j}=-1$ can appear. If at some point, one has $\delta_{j}=1$
then the multiplicity of $\nu\left(S_{j}\right)$ is written $\alpha p+\beta q$
for some $\alpha,\ \beta$ in $\mathbb{Z}.$ Thus it is even and $p_{j}$
cannot be equal to $-1.$ Therefore, $D_{2}$ and $D_{1}$ belongs
to the same connected component $\overline{D},$ which inductively
proved the proposition since $d^{E}$ is attached to $D_{2}.$
\end{itemize}
\item Suppose finally that $\delta_{1}=0$. One has $\delta_{2}=1.$ If
$p_{1}>0$ then the proposition is proved inductively. If not, two
cases may occur :
\begin{itemize}
\item if $n=2$ then $\nu\left(S_{2}\right)=\nu\left(S_{1}\right)$ is odd.
The proposition is proved applying it inductively to $S_{2}$ with
the sequence 
\[
0,\delta_{3},\ldots.
\]
The arguments are the same as above noticing that 
\[
\left[\frac{\nu\left(S_{2}\right)}{2}\right]=\left[\frac{\nu\left(S_{2}\right)-\delta_{2}}{2}\right].
\]
\item if $n\geq3$ and $p_{1}=0$ then $n=3$, $p$ is odd and $q$ is even.
The proposition is proved applying it inductively to $S_{2}$ with
the sequence 
\[
0,\ 1,\ \delta_{4}\ldots.
\]
Again, the arguments are the same as before. 
\end{itemize}
\end{itemize}
\end{proof}
Now, we introduce a foliation associated to $S_{d}$ prescribing some
topological data.
\begin{defn}
The numbered dual tree $\mathbb{A}\left[\mathcal{F}\right]$ of a
foliation $\mathcal{F}$ is a numbered graph constructed as follows.
Let $E$ be the minimal desingularization of $\mathcal{F}.$ The vertices
of $\mathbb{A}\left[\mathcal{F}\right]$ are in one-to-one correspondence
with the irreducible components of the exceptional divisor of $E.$
There is an edge between $D_{i}$ and $D_{j}$ is and only if $D_{i}\cap D_{j}\neq\emptyset$.
Each vertex is numbered following the next rule:
\begin{itemize}
\item if $E^{*}\mathcal{F}$ is dicritical along $D_{i}$, then $D_{i}$
is numbered $+\infty$ 
\item else it is numbered by the number of irreducible invariant curves
of $E^{*}\mathcal{F}$ intersecting $D_{i}$ transversely.
\end{itemize}
\end{defn}

Now, the proposition below produces the checked foliation. 
\begin{prop}
Let $\mathbb{A}$ the dual tree of $S_{d}$ numbered the following
way:
\begin{itemize}
\item if $p_{i}=-1$ then $D_{i}$ is numbered $\infty$.
\item if not, $D_{i}$ is numbered $p_{i}+\left(\textup{ the number of component of \ensuremath{d^{E}} meeting \ensuremath{D_{i}}}\right)$
\item $D_{N}$ is numbered $+\infty.$
\end{itemize}
There exists a foliation $\mathcal{F}\left[S_{d}\right]$ whose singularities
are linearizable and such that 
\[
\mathbb{A}\left[\mathcal{F}\left[S_{d}\right]\right]=\mathbb{A}.
\]
\end{prop}

\begin{proof}
Using a result of Lins-Neto \cite{alcides} whose statement is also
mentioned in \cite{Loray} and written in a more compact way. For
the arguments to come, we will refer to the latter version. 

The statement of Lins-Neto is quite long to enunciate because the
hypothesis require that we prescribe all the local and semi-local
data attached to the desired foliation. Below, to be the most specific
as possible, we will follow the numbering of the hypothesis in \cite{Loray}
p. 151. We require that 
\begin{itemize}
\item \textbf{Hypothesis} $\left(1\right)$: the desingularization of $\mathcal{F}\left[S_{d}\right]$
has the same topology as the desingularization of $S_{d}$. For the
sake of simplicity, we keep denoting by $D=\bigcup_{i=1}^{N}D_{i}$
the exceptional divisor of its desingularization.
\item \textbf{Hypothesis} $\left(2\right)$: $\mathcal{F}\left[S_{d}\right]$
is dicritical and regular along $D_{N}$. If $p_{i}=-1$ , then $\mathcal{F}\left[S_{d}\right]$
is dicritical and regular along $D_{i}$. If not, $D_{i}$ is invariant. 
\item \textbf{Hypothesis} $\left(5\right)$: At each corner point of $D$
that does not meet a dicritical component, $\mathcal{F}\left[S_{d}\right]$
admits a linear singularity written in some local coordinates $\left(x,y\right)$
\begin{equation}
\lambda x\dd y+y\dd x,\ \lambda\notin\mathbb{Q}^{-}\label{eq:eq11}
\end{equation}
where $xy=0$ is a local equation of $D$.
\item \textbf{Hypothesis} $\left(4\right),\left(6\right)$: For each $D_{i}$
with $p_{i}\geq0$, $\mathcal{F}\left[S_{d}\right]$ admits $p_{i}$
more linear singularities along $D_{i}$ and written in some local
coordinates $\left(x,y\right)$ 
\begin{equation}
\lambda x\dd y+y\dd x,\ \lambda\notin\mathbb{Q}^{-}\label{eq:eq12}
\end{equation}
where $x=0$ is a local equation of $D_{i}$.\\
The local analytic class of the singularities added above depends
on the value of $\lambda$ which is called the \emph{Camacho-Sad index
\cite{separatrice}} of the singularity $s$ along $D.$ It is denoted
by 
\[
\lambda=CS_{s}\left(\mathcal{F}\left[S_{d}\right],D\right)
\]
where $s$ is the singularity. Finally, for each component of $d^{E}$
attached to $D_{j}$ with $p_{j}\geq0$, $\mathcal{F}\left[S_{d}\right]$
admits one more linear singularity along $D_{j}.$
\end{itemize}
The remain hypothesis controle the projective representations of holonomy
of the desired foliation: this part is irrelevant for our construction
and can be chosen arbitrarily. 

The above data must satisfy some compatibility conditions stated in
the theorem of Lins-Neto : 
\begin{itemize}
\item two dicritical components cannot meet which is ensured by the second
property of Proposition \ref{lem-combi}. 
\item the Camacho-Sad indexes of the singularities along a given component
$D_{j}$ have to satisfy a relation known as \emph{the Camacho-Sad
relation} 
\[
\sum_{s\in D_{j}}CS_{s}\left(\mathcal{F}\left[S_{d}\right],D_{j}\right)=-D_{j}\cdot D_{j}.
\]
The third property in Proposition \ref{lem-combi} allows us to choose
the Camacho-Sad indices of the linear singularities added at (\ref{eq:eq11})
and at (\ref{eq:eq12}) in order to ensure the Camacho-Sad relation
for any component $D_{j}$. 
\end{itemize}
According to the theorem of Lins-Neto, there exists a germ of foliation
$\mathcal{F}\left[S_{d}\right]$ defined at the origin of $\left(\mathbb{C}^{2},0\right)$
that realizes all the above prescription. In particular, by construction,
one has 
\[
\mathbb{A}\left[\mathcal{F}\left[S_{d}\right]\right]=\mathbb{A}.
\]
\end{proof}
A lot of foliations can be constructed as above, prescribing freely
the projective representations of holonomy. Hence, there is a big
number of non analytically equivalent choices. However, all the foliations
build the way above share some properties. In any case, $\mathcal{F}\left[S_{d}\right]$
is dicritical along $D_{N}$. Its singularities are all linearizable
and thus $\mathcal{F}\left[S_{d}\right]$ is of \emph{second kind
}as defined in\emph{ \cite{MatQuasi,Genzmer1}}. Its desingularization
has the same topological type as the desingularization of $S_{d}$.
Moreover, the foliation $\mathcal{F}\left[S_{d}\right]$ is tangent
to some curve $\mathfrak{S}$ topologically equivalent to $S_{d}$
since $\mathfrak{S}$ and $S_{d}$ share the same process of desingularization.
Finally, the algebraic multiplicity is the desired one. Indeed, one
has the following result :

\begin{center}
\begin{figure}
\begin{centering}
\includegraphics[scale=0.25]{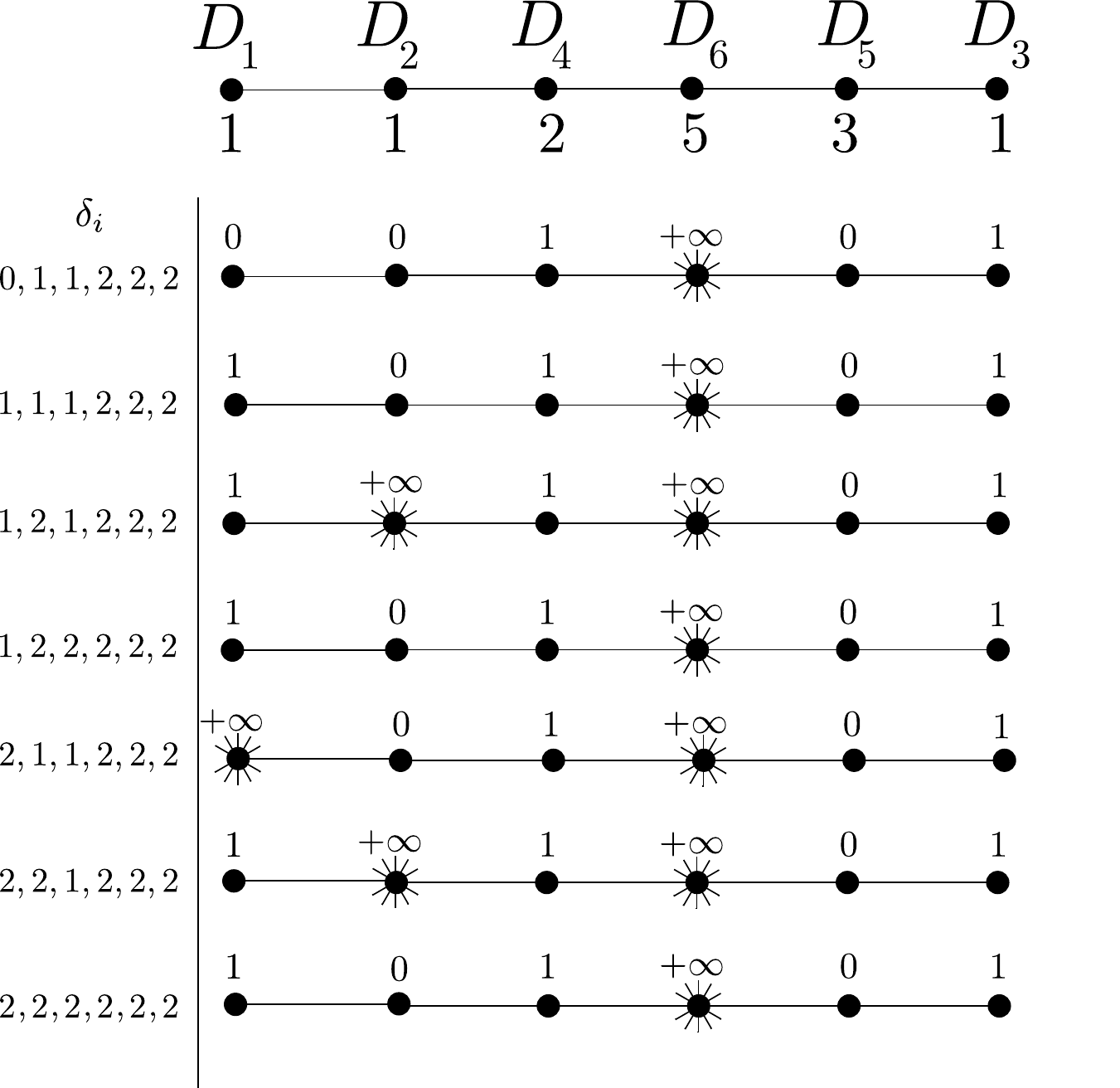}
\par\end{centering}
\bigskip{}

\caption{Dual numbered tree $\mathbb{A}\left[\mathcal{F}\left[S_{d}\right]\right]$
for $S=\left\{ y^{5}=x^{13}\right\} $ and any direction $d$.}
\end{figure}
\par\end{center}
\begin{lem}
Regardless the foliation $\mathcal{F}\left[S_{d}\right]$ constructed
as above, one has 
\[
\nu\left(\mathcal{F}\left[S_{d}\right]\right)=\left[\frac{\nu\left(S_{d}\right)}{2}\right].
\]
\end{lem}

\begin{proof}
Following a formula of Hertling in \cite{hertlingfor} - see Theorem
3.(a) - gives us 
\[
\nu\left(\mathcal{F}\left[S_{d}\right]\right)=\sum_{i=1}^{N-1}p_{i}e_{1i}+\delta_{1}-1.
\]
In the notations of the Hertling's formula, one has $\rho_{i}=e_{1i}$
and $\epsilon_{i}^{\left(k\right)}=p_{i}.$ Since $\nu\left(S_{N}\right)=1$
and $\delta_{N}=2$, one has $p_{N}=0.$ Using the expression of $\mathcal{E}^{-1}$
to invert the formula (\ref{eq:SuperRelation}), the first row yields
\begin{eqnarray*}
\sum_{i=1}^{N-1}p_{i}e_{1i}+\delta_{1}-1 & = & \left[\frac{\nu\left(S_{1}\right)-\delta_{1}}{2}\right]+\delta_{1}=\left[\frac{\nu\left(S_{d}\right)}{2}\right].
\end{eqnarray*}
\end{proof}

\subsection{\label{sec:Cohomology-study-of}Deformations of $\mathcal{F}\left[S_{d}\right]$.}

In this section, we are interested in the deformations of foliations
with a cohomological approach. 

\subsubsection{Basic vector fields and deformations. }

Let $\omega$ be a germ of $1-$form and $X$ a germ of vector field.
The vector field $X$ is said to be \emph{basic} for $\omega$ if
and only if 
\[
\left(L_{X}\omega\right)\wedge\omega=d\left(\omega\left(X\right)\right)\wedge\omega-\omega\left(X\right)d\omega=0.
\]

The property of being basic for the $1-$form $\omega$ depends only
on the foliation induced by $\omega,$ since for any function $f$,
one has 
\[
L_{X}\left(f\omega\right)\wedge f\omega=f^{2}\left(L_{X}\omega\right)\wedge\omega.
\]

\begin{lem}
Let $X$ be a germ of vector field. It is basic for $\omega$ if and
only if for any $t\in\left(\mathbb{C},0\right),$ the flow at time
$t$ of $X$, denoted by $e^{\left[t\right]X}$, is an automorphism
of $\omega$, i.e.,
\[
\left(\left(e^{\left[t\right]X}\right)^{*}\omega\right)\wedge\omega=0.
\]
\end{lem}

In particular, the flow $e^{\left[t\right]X}$ preserves the leaves
of the foliation - but may exchange it. 

More generally, a germ of automorphism of $\omega$ is a germ of diffeomorphism
$\phi$ such that $\left(\phi^{*}\omega\right)\wedge\omega=0.$ If
$\phi$ is tangent to $\textup{Id},$ then there exists a\emph{ }formal
basic vector field $X$ such that $e^{\left[1\right]X}=\phi.$ In
what follows, we will simply denote the flow at time $1$ of $X$
by $e^{X}.$ If $X$ is singular at $p$, then the flow $e^{X}$ is
convergent in a neighborhood of $p.$ 

Thanks to basic automorphisms, we can describe a surgery construction
that produces many non-equivalent germs of foliations from a given
one. Consider the desingularization $E:\left(\mathcal{M},D\right)\to\left(\mathbb{C}^{2},0\right)$
of some singular foliation $\mathcal{F}$ at $\left(\mathbb{C}^{2},0\right)$.
For any covering $\left\{ U_{i}\right\} _{i\in I}$ of a neighborhood
of $D$ in $\mathcal{M}$ and for any $2-$intersection $U_{ij}=U_{i}\cap U_{j}$,
we consider $\phi_{ij}$ a basic automorphism of $E^{*}\mathcal{F}$
which is the identity map along $U_{ij}\cap D$. We suppose that the
family $\left\{ \phi_{ij}\right\} _{i,j}$ satisfies the cocycle relation:
on any $3$-intersection $U_{ijk},$ one has 
\[
\phi_{ij}\circ\phi_{jk}\circ\phi_{ki}=\textup{Id}.
\]
We construct a\emph{ }manifold with the following gluing 
\[
\mathcal{M}\left[\phi_{ij}\right]=\coprod_{i}U_{i}/_{\left(x,i\right)\sim\left(\phi_{ij}\left(x\right),j\right)}
\]
which is a neighborhood of some divisor isomorphic to $D.$ This manifold
is foliated by a foliation $\mathcal{F}^{\prime}$ obtained by gluing
with the same collection of maps the family of restricted foliations
$\left\{ \left.E^{*}\mathcal{F}\right|_{U_{i}}\right\} _{i}$.
\begin{lem}
There exists a germ of singular foliation at the origin of $\left(\mathbb{C}^{2},0\right)$
denoted by $\mathcal{F}\left[\phi_{ij}\right]$ and a process of blowing-ups
$E^{\prime}$ such that $\left(E^{\prime}\right)^{*}\mathcal{F}\left[\phi_{ij}\right]$
is analytically equivalent to $\mathcal{F}^{\prime}.$
\end{lem}

\begin{proof}
The manifold $\mathcal{M}\left[\phi_{ij}\right]$ is an open neighborhood
of a divisor whose intersection matrix is the same as the one of $D$
since the gluing map $\phi_{ij}$ let invariant the trace of the divisor
$U_{ij}\cap D.$ In particular, its intersection matrix is definite
negative. Following the Grauert's contraction result \cite{grauerthans},
there exists a process of blowing-ups $E^{\prime}:\left(\mathcal{M}^{\prime},D^{\prime}\right)\to\left(\mathbb{C}^{2},0\right)$
such that $\mathcal{M}^{\prime}$ is analytically equivalent to $\mathcal{M}\left[\phi_{ij}\right]$.
Being analytically equivalent to $\mathcal{M}\left[\phi_{ij}\right]$,
the manifold $\mathcal{M}^{\prime}$ is foliated. Since $E^{\prime}$
is an isomorphism between $\mathcal{M}^{\prime}\setminus D^{\prime}$
and $\left(\mathbb{C}^{2},0\right)\setminus\left\{ 0\right\} $, there
exists a foliation in $\left(\mathbb{C}^{2},0\right)\setminus\left\{ 0\right\} $
whose pull-back by $E^{\prime}$ coincides with the foliation of $\mathcal{M}^{\prime}$
on $\mathcal{M}^{\prime}\setminus D^{\prime}$. The Hartogs's extension
result allows us to extend this foliation in $\left(\mathbb{C}^{2},0\right).$
The obtained foliation is $\mathcal{F}\left[\phi_{ij}\right]$. 
\end{proof}
A foliation build the way above is said to be a \emph{basic surgery}
of $\mathcal{F}.$ Our goal is to study the basic surgeries of $\mathcal{F}\left[S_{d}\right]$
and in particular to prove the following 
\begin{prop}
\label{thm:Any-curve-C}For any curve $\mathfrak{S}$ topologically
equivalent to $S_{d}$, there is a $1$-form $\omega\in\Omega^{1}\left(\mathfrak{S}\right)$
defining a foliation obtained from a basic surgery of $\mathcal{F}\left[S_{d}\right].$
\end{prop}

The proof is based upon the study of deformations of $\mathcal{F}\left[S_{d}\right]$
with a cohomological point of view, that is developed below. 

\subsubsection{The sheaf $TS_{d}$.}

In the desingularization $E:\left(\mathcal{M},D\right)\to\left(\mathbb{C}^{2},0\right)$,
let us consider the sheaf $TS_{d}$, with $D$ as basis, of vector
fields tangent to $D$ and to $S^{E}$ that vanish along the strict
transform $d^{E}$.

For any divisor $\Sigma=\sum n_{i}\Sigma_{i}$ in $\mathcal{M}$,
we denote by $\Omega^{2}\left(\Sigma\right)$ the sheaf with $D$
as basis, of $2-$forms $\omega$ such that the multiplicity of $\omega$
along $\Sigma_{i}$ satisfies
\[
\nu_{\Sigma_{i}}\left(\omega\right)\geq-n_{i}.
\]

Let $F$ be a \emph{balanced} equation of $\mathcal{F}\left[S_{d}\right]$
as defined in \cite{Genzmer1}. First, we prove the following proposition. 
\begin{prop}
\label{proposition-fondamentale}In Cech cohomology, one has 
\[
H^{1}\left(D,\Omega^{2}\left(2\left(F\right)^{E}-S_{d}^{E}+\overline{D}\right)\right)=0
\]
where the divisor $\left(F\right)^{E}$ is $\left(F=0\right)^{E}-\left(F=\infty\right)^{E}$
and $\overline{D}$ is the divisor $D$ deprived of $D_{N}$ and of
the components $D_{i}$ for which $p_{i}=-1.$
\end{prop}

The proof is an induction on the length of the desingularization $E.$
The first step is the following lemma.

\global\long\def\EQ#1#2{\raisebox{.65ex}#1\raisebox{-.65ex}#2}%

\begin{lem}
\label{lemma1}Let us consider a germ of divisor $\Sigma$ at the
origin of $\left(\mathbb{C}^{2},0\right)$. Let $E_{1}:\left(\mathcal{M}_{1},D_{1}\right)\to\left(\mathbb{C}^{2},0\right)$
be the standard blowing-up of the origin. Then, for any $n\geq0,$
the following are equivalent 

\begin{itemize}
\item The multiplicity of $\Sigma$ at the origin satisfies $\nu\left(\Sigma\right)\geq n.$ 
\item The first cohomology group of $\Omega^{2}\left(\Sigma^{E_{1}}+nD_{1}\right)$
on $D_{1}$ vanishes 
\begin{equation}
H^{1}\left(D_{1},\Omega^{2}\left(\Sigma^{E}+nD_{1}\right)\right)=0.\label{eq:cohonul}
\end{equation}
\end{itemize}
\end{lem}

\begin{proof}
Let $l$ be an equation of $\Sigma$. Consider the standard coordinates
of the blowing-up together with its standard covering. 
\[
U_{1}:\begin{cases}
y=y_{1}x_{1}\\
x=x_{1}
\end{cases}\qquad U_{2}:\begin{cases}
y=y_{2}\\
x=y_{2}x_{2}
\end{cases}.
\]
The global sections of $\Omega^{2}\left(\Sigma^{E_{1}}+nD_{1}\right)$
on each associated open sets are written 
\begin{eqnarray*}
\Omega^{2}\left(\Sigma^{E_{1}}+nD_{1}\right)\left(U_{1}\right) & = & \left\{ \left.f\left(x_{1},y_{1}\right)\frac{1}{l_{1}x_{1}^{n}}\dd x_{1}\wedge\dd y_{1}\right|f\in\mathcal{O}\left(U_{1}\right)\right\} \\
\Omega^{2}\left(\Sigma^{E_{1}}+nD_{1}\right)\left(U_{2}\right) & = & \left\{ \left.g\left(x_{2},y_{2}\right)\frac{1}{l_{2}y_{2}^{n}}\dd x_{2}\wedge\dd y_{2}\right|g\in\mathcal{O}\left(U_{2}\right)\right\} \\
\Omega^{2}\left(\Sigma^{E_{1}}+nD_{1}\right)\left(U_{1}\cap U_{2}\right) & = & \left\{ \left.h\left(x_{1},y_{1}\right)\frac{1}{l_{1}x_{1}^{n}}\dd x_{1}\wedge\dd y_{1}\right|h\in\mathcal{O}\left(U_{1}\cap U_{2}\right)\right\} 
\end{eqnarray*}
where $l_{1}=\frac{l\circ E_{1}}{x_{1}^{\nu\left(\Sigma\right)}}$,
$l_{2}=\frac{l\circ E_{1}}{y_{2}^{\nu\left(\Sigma\right)}}$. Since
the covering $\left\{ U_{1},U_{2}\right\} $ is acyclic, one has the
following isomorphism
\[
H^{1}\left(D_{1},\Omega^{2}\left(\Sigma^{E_{1}}+nD_{1}\right)\right)\simeq\frac{\Omega^{2}\left(\Sigma^{E_{1}}+nD_{1}\right)\left(U_{1}\cap U_{2}\right)}{\Omega^{2}\left(\Sigma^{E_{1}}+nD_{1}\right)\left(U_{1}\right)\oplus\Omega^{2}\left(\Sigma^{E_{1}}+nD_{1}\right)\left(U_{2}\right)}.
\]

Therefore, the dimension of (\ref{eq:cohonul}) is the number of obstructions
to the following cohomological equation
\begin{multline*}
h\left(x_{1},y_{1}\right)\frac{1}{l_{1}x_{1}^{n}}\dd x_{1}\wedge\dd y_{1}=g\left(x_{2},y_{2}\right)\frac{1}{l_{2}y_{2}^{n}}\dd x_{2}\wedge\dd y_{2}\\
-f\left(x_{1},y_{1}\right)\frac{1}{l_{1}x_{1}^{n}}\dd x_{1}\wedge\dd y_{1}
\end{multline*}
which is equivalent to 
\begin{equation}
h\left(x_{1},y_{1}\right)=-f\left(x_{1},y_{1}\right)-\frac{1}{y_{1}^{-\nu\left(\Sigma\right)+n+1}}g\left(\frac{1}{y_{1}},y_{1}x_{1}\right).\label{eq:coho}
\end{equation}
Let $h=x_{1}^{i_{0}}y_{1}^{j_{0}}$$.$ Then $h$ is an obstruction
to (\ref{eq:coho}) if and only if $j_{0}<0$ and the following system
cannot be solved in $\mathbb{N}^{2}$
\[
\begin{cases}
i_{0}=j\\
j_{0}=j-i+\nu\left(\Sigma\right)-n-1
\end{cases}\Longleftrightarrow\begin{cases}
j=i_{0}\\
i=i_{0}-j_{0}+\nu\left(\Sigma\right)-n-1
\end{cases}.
\]
Thus, $\nu\left(\Sigma\right)\geq n$ if and only if there is no obstruction. 
\end{proof}
Now let us prove Proposition \ref{proposition-fondamentale}.
\begin{proof}
The proof of the proposition is an induction on the length of the
desingularization of $S_{d}.$ Let us write 
\[
E=E_{1}\circ E^{2}.
\]
Let $U_{1}$ be $D_{1}\setminus\textup{Sing}\left(S_{2}\right)$ and
$U_{2}$ a very small neighborhood of $\textup{Sing}\left(S_{2}\right)$
as in the proof of Proposition \ref{dimension}. We defined the following
open sets 
\begin{equation}
\mathcal{U}_{1}=\left(E^{2}\right)^{-1}\left(U_{1}\right)\qquad\mathcal{U}_{2}=\left(E^{2}\right)^{-1}\left(U_{2}\right)\label{eq:covering}
\end{equation}
The system $\left\{ \mathcal{U}_{1},\mathcal{U}_{2}\right\} $ is
an open covering of $D$. The associated Mayer-Vietoris sequence for
the sheaf $\Omega^{2}\left(2\left(F\right)^{E}-S_{d}^{E}+\overline{D}\right)$
is written 

\begin{multline}
H^{0}\left(\mathcal{U}_{1},\Omega^{2}\left(2\left(F\right)^{E}-S_{d}^{E}+\overline{D}\right)\right)\bigoplus H^{0}\left(\mathcal{U}_{2},\Omega^{2}\left(\cdots\right)\right)\\
\to H^{0}\left(\mathcal{U}_{1}\cap\mathcal{U}_{2},\Omega^{2}\left(\cdots\right)\right)\to\mathcal{N}\to0\label{eq:exact-cool}
\end{multline}

and 
\begin{multline}
0\to\mathcal{N}\to H^{1}\left(D,\Omega^{2}\left(\cdots\right)\right)\to H^{1}\left(\mathcal{U}_{1},\Omega^{2}\left(\cdots\right)\right)\bigoplus H^{1}\left(\mathcal{U}_{2},\Omega^{2}\left(\cdots\right)\right).\label{eq:autre}
\end{multline}
We are going to identify each term of the above exact sequences.

The manifold $D_{1}\setminus\textup{Sing}\left(S_{2}\right)$ is isomorphic
to $\mathbb{C}$. Thus, it is a Stein. Since, the sheaf $\Omega^{2}\left(\cdots\right)$
is coherent, its cohomology vanishes on $\mathcal{U}_{1}$ \cite{grauertremmert}
and, in (\ref{eq:autre}), the following relation holds, 
\[
H^{1}\left(\mathcal{U}_{1},\Omega^{2}\left(2\left(F\right)^{E}-S_{d}^{E}+\overline{D}\right)\right)=0.
\]
Let $\mathcal{F}_{2}$ be defined by the germ of foliation $E_{1}^{*}\mathcal{F}\left[S_{d}\right]$
at $\textup{Sing}\left(S_{2}\right)$. By construction, the foliation
$\mathcal{F}_{2}$ let invariant $S_{2}.$ Let $F_{2}$ be a balanced
equation of $\mathcal{F}_{2}$. Let $h$ be a local equation of $D_{1}$
at $\textup{Sing}\left(S_{2}\right).$ Two cases have to be considered

\begin{itemize}
\item If $D_{1}$ is invariant for $\mathcal{F}\left[S_{d}\right]$, then,
following \cite{Genzmer1}, $F_{2}$ can be chosen so that 
\[
\left(F_{2}\right)^{E^{2}}=\left(h\right)^{E^{2}}+\left.\left(F\right)^{E}\right|_{\mathcal{U}_{2}}
\]
Thus, if the direction $d_{2}$ of $S_{2}$ is chosen to be the local
trace at $\textup{Sing}\left(S_{2}\right)$ of the union of $d^{E_{1}}$
and $D_{1}$, then the next equalities hold 
\begin{align*}
\left.\left(2\left(F\right)^{E}-S_{d}^{E}+\overline{D}\right)\right|_{\mathcal{U}_{2}} & =2\left(\left(F_{2}\right)^{E^{2}}-\left(h\right)^{E^{2}}\right)-\left.S_{d}^{E}\right|_{\mathcal{U}_{2}}+\left.\overline{D}\right|_{\mathcal{U}_{2}}\\
 & =2\left(F_{2}\right)^{E^{2}}-2\left(h\right)^{E^{2}}-\left.S_{d}^{E}\right|_{\mathcal{U}_{2}}+\overline{D^{2}}+\left(h\right)^{E^{2}}\\
 & =2\left(F_{2}\right)^{E^{2}}-S_{2,d_{2}}^{E^{2}}+\overline{D^{2}}
\end{align*}
\item If $D_{1}$ is not invariant for $\mathcal{F}\left[S_{d}\right]$
then $F_{2}$ can be chosen so that 
\[
\left(F_{2}\right)^{E^{2}}=\left(F\right)^{E}
\]
Thus setting for the direction $d_{2}$ of $S_{2}$ the local trace
at $\textup{Sing}\left(S_{2}\right)$ of the sole $d^{E_{1}}$ still
yields 
\[
\left.\left(2\left(F\right)^{E}-S_{d}^{E}+\overline{D}\right)\right|_{\mathcal{U}_{2}}=2\left(F_{2}\right)^{E^{2}}-S_{2,d_{2}}^{E^{2}}+\overline{D^{2}}
\]
since here $\left.\overline{D}\right|_{\mathcal{U}_{2}}=\overline{D^{2}}$. 
\end{itemize}
In any case, applying inductively Proposition \ref{proposition-fondamentale}
to $S_{2}$ and to the associated divisor $2\left(F_{2}\right)^{E^{2}}-S_{2,d_{2}}^{E^{2}}+\overline{D^{2}}$
ensures that, in (\ref{eq:autre}), one has 
\[
H^{1}\left(\mathcal{U}_{2},\Omega^{2}\left(2\left(F\right)^{E}-S_{d}^{E}+\overline{D}\right)\right)=H^{1}\left(\mathcal{U}_{2},\Omega^{2}\left(2\left(F_{2}\right)^{E^{2}}-S_{2,d_{2}}^{E^{2}}+\overline{D^{2}}\right)\right)=0.
\]
The map $E^{2}$ induces isomorphisms in cohomology 
\begin{eqnarray}
H^{0}\left(\mathcal{U}_{1},\Omega^{2}\left(2\left(F\right)^{E}-S_{d}^{E}+\overline{D}\right)\right) & \simeq & H^{0}\left(U_{1},\Omega^{2}\left(2\left(F\right)^{E_{1}}-S_{d}^{E_{1}}+\overline{D_{1}}\right)\right)\nonumber \\
H^{0}\left(\mathcal{U}_{1}\cap\mathcal{U}_{2},\Omega^{2}\left(\cdots\right)\right) & \simeq & H^{0}\left(U_{1}\cap U_{2},\Omega^{2}\left(\cdots\right)\right).\label{eq:cohozero}
\end{eqnarray}
Let us prove that $E^{2}$ induces also an isomorphism on the set
of global sections along $U_{2}$ and $\mathcal{U}_{2}$. If $\eta$
is a global section of $\Omega^{2}\left(2\left(F\right)^{E}-S_{d}^{E}+\overline{D}\right)$
on $\mathcal{U}_{2}$ then the push-forward of $\eta$ by $E^{2}$
can be extended analytically at $\textup{Sing}\left(S_{2}\right)$
by Hartogs's extension result. It induces naturally a section of $\Omega^{2}\left(2\left(F\right)^{E_{1}}-S_{d}^{E_{1}}+\overline{D_{1}}\right)$
on $U_{2}.$ Thus, $E^{2}$ induces a injective map
\begin{equation}
H^{0}\left(\mathcal{U}_{2},\Omega^{2}\left(2\left(F\right)^{E}-S_{d}^{E}+\overline{D}\right)\right)\overset{E^{2}}{\hookrightarrow}H^{0}\left(U_{2},\Omega^{2}\left(2\left(F\right)^{E_{1}}-S_{d}^{E_{1}}+\overline{D_{1}}\right)\right).\label{eq:142-1}
\end{equation}

By induction, it is enough to prove that (\ref{eq:142-1}) is onto
when $E^{2}$ is the simple blowing-up of $\textup{Sing}\left(S_{2}\right)$
and $D$ reduced to $D_{1}\cup D_{2}$. \\
Let $\eta$ be a section of $\Omega^{2}\left(2\left(F\right)^{E_{1}}-S_{d}^{E_{1}}+\overline{D_{1}}\right)$
on $U_{2}$.

\begin{itemize}
\item If $D_{1}$ is not dicritical for $\mathcal{F}\left[S_{d}\right]$
then $\eta$ is written in coordinates
\[
\eta=hf\frac{\dd x\wedge\dd y}{x},
\]
where $x$ is a local equation of $D_{1},$ $f$ is any meromorphic
function whose local divisor is $2\left(F\right)^{E_{1}}-S_{d}^{E_{1}}$
and $h$ is any holomorphic function. If $\delta_{2}=1$ then the
possible component of $d$ meets $D_{1}$ at a different point from
$\textup{Sing}\left(S_{2}\right)$. Thus the valuation of $f$ is
equal to 
\[
\nu\left(f\right)=e_{2n}-2\sum_{i=2}^{n-1}p_{i}e_{2i}=e_{2n}-2\left[\frac{e_{2n}-1}{2}\right]-2\geq-1
\]
Now, after the blowing-up $E^{2}$ which is written in adapted coordinates
$E^{2}\left(x,t\right)=\left(x,tx\right),$ the pull back of $\eta$
is written 
\[
E^{2*}\eta=h^{*}f^{*}\dd x\wedge\dd t.
\]
Thus, the valuation of $E^{2*}\eta$ along $D_{2}$ is at least $-1.$
The exceptional divisor of $E^{2}$ cannot be dicritical for $\mathcal{F}\left[S_{d}\right]$
since $\delta_{2}=1$. Therefore, $E^{2*}\eta$ is a section of $\Omega^{2}\left(2\left(F\right)^{E^{2}}-S_{d}^{E^{2}}+\overline{D_{1}\cup D_{2}}\right)$
along $D_{1}\cup D_{2}$. Now, if $\delta_{2}=2$ then one of the
components of $d^{E_{1}}$, say $d_{1}^{E_{1}},$ meets $\textup{Sing}\left(S_{2}\right)$.
Whether or not the component $d_{1}^{E}$ meets a dicritical component,
the valuation of $f$ is at least 
\[
\nu\left(f\right)\geq e_{2n}-2\sum_{i=2}^{n-1}p_{i}e_{2i}-1=e_{2n}-2\left[\frac{e_{2n}}{2}\right]-1
\]
If the exceptional divisor of $E^{2}$ is dicritical then $e_{2n}$
is odd and $\nu\left(f\right)\geq0.$ If not, $\nu\left(f\right)\geq-1.$
Thus, wether the exceptional divisor of $E^{2}$ is dicritical or
not, $E^{2*}\omega$ is a section of $\Omega^{2}\left(2\left(F\right)^{E^{2}}-S_{d}^{E^{2}}+\overline{D_{1}\cup D_{2}}\right)$
along $D_{1}\cup D_{2}$.
\item if $D_{1}$ is dicritical then $\delta_{2}=1$. Moreover, $\eta$
is written 
\[
\eta=hf\dd x\wedge\dd y,\qquad E^{2*}\eta=h^{*}f^{*}x\dd x\wedge\dd t
\]
where 
\[
\nu\left(f\right)+1=e_{2n}-\sum_{i=2}^{n-1}p_{i}e_{2i}+1=e_{2n}-1-2\left[\frac{e_{2n}-1}{2}\right]\geq0.
\]
Hence, $E^{2*}\omega$ is still a section of $\Omega^{2}\left(2\left(F\right)^{E^{2}}-S_{d}^{E^{2}}+\overline{D_{1}\cup D_{2}}\right)$
along $D_{1}\cup D_{2}$.
\end{itemize}
By induction on the length of $E^{2}$, the isomorphism (\ref{eq:142-1})
is proved. Thus, the isomorphisms (\ref{eq:cohonul}) and the exact
sequence (\ref{eq:exact-cool}) identify $\mathcal{N}$ with the cohomology
group{\tiny{} }\textbf{
\[
H^{1}\left(D_{1},\Omega^{2}\left(2\left(F\right)^{E_{1}}-S_{d}^{E_{1}}+\overline{D_{1}}\right)\right).
\]
}Let us prove that the latter vanishes. If $p_{1}=-1$, then $D_{1}$
is dicritical and $\delta_{1}=2$ and $\delta_{2}=1$. Therefore,
\[
\nu\left(2\left(F\right)^{E_{1}}-S_{d}^{E_{1}}\right)=e_{1n}-2\sum_{i=2}^{n-1}p_{i}e_{1i}=e_{1n}-2\left(\left[\frac{e_{1n}-2}{2}\right]+2\right)=-1
\]
since $e_{1n}$ is odd. If $p_{1}\neq-1,$ then
\[
\nu\left(2\left(F\right)^{E_{1}}-S_{d}^{E_{1}}\right)=e_{1n}-2\sum_{i=1}^{n-1}p_{i}e_{1i}-\delta_{1}=e_{1n}-\delta_{1}-2\left[\frac{e_{1n}-\delta_{1}}{2}\right]-2\leq-1.
\]
Therefore, according to Lemma \ref{lemma1}, $\mathcal{N}$ vanishes,
which completes the proof of Proposition \ref{proposition-fondamentale}.
\end{proof}
To compare the deformations of $\mathcal{F}\left[S_{d}\right]$ and
of the underlying curve $S_{d}$, we introduce the following operator. 
\begin{defn}
The operator of basic vector fields for $\mathcal{F}\left[S_{d}\right]$
is a morphism of sheaves defined by 
\begin{equation}
\mathcal{B}:X\in T\mathcal{S}_{d}\mapsto L_{X}E^{*}\frac{\omega}{F}\wedge E^{*}\frac{\omega}{F}\in\Omega^{2}\label{eq:142}
\end{equation}
\end{defn}

where $\omega$ is any $1-$form with an isolated singularity defining
$\mathcal{F}\left[S_{d}\right]$ and $F$ any balanced equation of
$\mathcal{F}\left[S_{d}\right]$. 

The operator of basic vector fields may behave quite wildly around
the singular point of $\mathcal{F}\left[S_{d}\right]$. Indeed, one
can check that the description of its local image at singular points
may involve the phenomenon known as \emph{small divisors}. However,
for our construction, we can disregard what happens exactly at the
singulart points, since \emph{we control everything happening around.}
To take into account this remark, we introduce the following notation
:

\textbf{Notation. }For any sheaf $\mathfrak{F}$ of basis $D$, we
denote by $\mathfrak{F}^{\circ}$ the sheaf whose stalk satisfies
that for all $x\in D\setminus\textup{Sing}\left(\mathcal{F}\left[S_{d}\right]\right)$,
$\left(\mathfrak{F}\right)_{x}=\left(\mathfrak{F}^{\circ}\right)_{x}$
and for all $x\in\textup{Sing}\left(\mathcal{F}\left[S_{d}\right]\right)$,
$\left(\mathfrak{F}^{\circ}\right)_{x}=0$. 

The interest of the above notation relies on the following lemma:
\begin{lem}
For any $i\geq1,$ one has 
\[
H^{i}\left(D,\mathfrak{F}\right)=H^{i}\left(D,\mathfrak{F}^{\circ}\right).
\]
\end{lem}

\begin{proof}
Indeed, there is a direct sum of skyscraper sheaves $\mathfrak{F}_{\circ}$
such that $\mathfrak{F}^{\circ}=\mathfrak{F}/\mathfrak{\mathfrak{F}_{\circ}}$.
The long exact of sheaves associated to the short sequence 
\[
0\to\mathfrak{F}_{\circ}\to\mathfrak{F}\to\mathfrak{F}/\mathfrak{\mathfrak{F}_{\circ}}\to0
\]
 and the fact that the cohomology of $\mathfrak{F}_{\circ}$vanishes
in degree more than $1$ ensure the lemma. 
\end{proof}
\begin{prop}
\label{prop:Infinitesimal}Let $\mathcal{B}_{n}\left(\mathcal{F}\left[S_{d}\right]\right)$
be the sheaf defined by the kernel 
\[
\mathcal{B}_{n}\left(\mathcal{F}\left[S_{d}\right]\right)=\ker\left(\left.\mathcal{B}\right|_{\mathfrak{M}^{n}\cdot TS_{d}}\right)
\]
where $\mathfrak{M}^{n}$ is the $n^{th}$ power of the sheaf of $\mathcal{O}-$module
generated by the functions $E^{*}f$ with $f\left(0\right)=0.$ There
is an exact sequence of sheaves 
\begin{equation}
0\to\mathcal{B}_{n}\left(\mathcal{F}\left[S_{d}\right]\right)^{\circ}\to\mathfrak{M}^{n}\cdot TS_{d}^{\circ}\to\mathfrak{M}^{n}\cdot\Omega^{2}\left(2\left(F\right)^{E}-S_{d}^{E}+\overline{D}\right)^{\circ}\to0\label{eq:ShortSeq}
\end{equation}
In particular, extracted from the long exact in cohomology associated
to \ref{eq:ShortSeq}, there is an exact sequence
\begin{equation}
H^{1}\left(D,\mathcal{B}_{n}\left(\mathcal{F}\left[S_{d}\right]\right)^{\circ}\right)\to H^{1}\left(D,\mathfrak{M}^{n}\cdot TS_{d}^{\circ}\right)\to0\label{eq:surj}
\end{equation}
 
\end{prop}

\begin{proof}
The first part of the proposition is a computation in local coordinates.
We describe the image of $\mathfrak{M}^{n}\cdot TS_{d}$ by the operator
$\mathcal{B}$. Since, $\mathcal{F}\left[S_{d}\right]$ is of second
kind \cite{MatQuasi}, the multiplicties of $\mathcal{F}\left[S_{d}\right]$
and of the balanced equation $F$ along any irreducible component
$D_{i}$ of the exceptional divisor satisfy \cite{Genzmer1}

\begin{itemize}
\item $\nu_{D_{i}}\left(\mathcal{F}\left[S_{d}\right]\right)=\nu_{D_{i}}\left(E^{*}F\right)$
if $D_{i}$ is dicritical 
\item $\nu_{D_{i}}\left(\mathcal{F}\left[S_{d}\right]\right)=\nu_{D_{i}}$$\left(E^{*}F\right)+1$
else.
\end{itemize}
Let $p$ be a regular point of $\mathcal{F}\left[S_{d}\right]$ where
the foliation tangent to exceptional divisor. In some local coordinates
$\left(x,y\right)$ around $p$, the pull-back $E^{*}\frac{\omega}{F}$
is written 
\[
E^{*}\frac{\omega}{F}=u\frac{\dd x}{x}
\]
where $x$ is a local equation of $D.$ Now, a local section $X$
of $\mathfrak{M}^{n}\cdot TS_{d}$ is written 
\[
X=x^{m}\left(ax\frac{\partial}{\partial x}+b\frac{\partial}{\partial y}\right),\ a,b\in\mathbb{C}\left\{ x,y\right\} .
\]
Therefore, applying the basic operator leads to 
\[
\mathcal{B}\left(X\right)=x^{m}u^{2}\frac{\partial a}{\partial y}\frac{\dd x\wedge\dd y}{x}
\]
which is a local section of $\mathfrak{M}^{n}\cdot\Omega^{2}\left(2\left(F\right)^{E}-S_{d}^{E}+\overline{D}\right).$
Since the equation $\frac{\partial a}{\partial y}=h$ can be solved
for any $h$, the operator $\mathcal{B}$ is onto locally around $p$.
This property is true for any type of regular points for $\mathcal{F}\left[S_{d}\right]$. 

The sheaf $\mathfrak{M}^{n}$ is generated by its global sections.
Therefore, Proposition \ref{proposition-fondamentale} ensures that
\[
H^{1}\left(D,\mathfrak{M}^{n}\cdot\Omega^{2}\left(2\left(F\right)^{E}-S_{d}^{E}+\overline{D}\right)\right)=0.
\]
Finaly, the long exact sequence in cohomology associated to (\ref{eq:ShortSeq})
proves the end of Proposition \ref{prop:Infinitesimal}.
\end{proof}

\subsubsection{Deformations of $\mathcal{F}\left[S_{d}\right]$ .\label{sec:proof-of-Theorem}}

Proposition \ref{prop:Infinitesimal} can be expressed as follows:
any infinitesimal deformation of $S_{d}$ tangent to $D$ at order
$n$ can be followed by an infinitesimal deformation of the foliation
$\mathcal{F}\left[S_{d}\right]$ at the same level of tangency. Roughly
speaking, the proof of Proposition \ref{thm:Any-curve-C} consists
in an \emph{non-commutative analog.} Actually, let us consider the
following sheaves of non-abelian groups
\begin{defn}
For any \emph{involutive} sub-sheaf $\mathfrak{I}$ of the sheaf of
tangent vector fields to $S_{d}^{E}$ that vanish along $d$ and $D$,
we consider 
\[
\mathfrak{G}\left(\mathfrak{I}\right)
\]
the sheaf of non-abelian groups generated by the flows of vector fields
in $\mathfrak{I}.$ 
\end{defn}

According to the Campbell-Hausdorff formula, 
\begin{equation}
e^{X}e^{Y}=e^{X+Y+\frac{1}{2}\left[X,Y\right]+\frac{1}{12}\left(\left[X,\left[X,Y\right]\right]-\left[Y,\left[X,Y\right]\right]\right)+\cdots}\label{eq:camp}
\end{equation}
any element of $\mathfrak{G}\left(\mathfrak{I}\right)$ is a flow
of an element of $\mathfrak{I}.$ 

The first step of the proof is the following:
\begin{prop}
\label{prop:Induced-in-non-abelian}Extracted from the long exact
sequence in cohomology induced by the embedding $\mathfrak{G}\left(\mathcal{B}_{1}\left(\mathcal{F}\left[S_{d}\right]\right)^{\circ}\right)\hookrightarrow\mathfrak{G}\left(\mathfrak{M}\cdot TS_{d}^{\circ}\right)$,
the following sequence 
\[
H^{1}\left(D,\mathfrak{G}\left(\mathcal{B}_{1}\left(\mathcal{F}\left[S_{d}\right]\right)^{\circ}\right)\right)\to H^{1}\left(D,\mathfrak{G}\left(\mathfrak{M}\cdot TS_{d}^{\circ}\right)\right)\to0
\]
is exact. 
\end{prop}

\begin{proof}
Let us consider a $1-$cocycle $\left\{ \phi_{ij}\right\} _{ij}\in\mathcal{Z}^{1}\left(D,\mathfrak{G}\left(\mathfrak{M}\cdot TS_{d}^{\circ}\right)\right).$
By definition, this is a flow 
\begin{equation}
\phi_{ij}=e^{X_{ij}}\label{eq:1}
\end{equation}
where $\left\{ X_{ij}\right\} _{ij}\in\mathcal{Z}^{1}\left(D,\mathfrak{M}\cdot TS_{d}^{\circ}\right)$.
By induction on $n$, we are going to prove that there exist $\left\{ B_{ij}^{n}\right\} _{ij}\in\mathcal{Z}^{1}\left(D,\mathcal{B}_{1}\left(\mathcal{F}\left[S_{d}\right]\right)^{\circ}\right)$,
$\left\{ X_{i}^{n}\right\} _{i}\in\mathcal{Z}^{0}\left(D,\mathfrak{M}\cdot TS_{d}^{\circ}\right)$
and $\left\{ X_{ij}^{n}\right\} _{ij}\in\mathcal{Z}^{1}\left(D,\mathfrak{M}^{n}\cdot TS_{d}^{\circ}\right)$
such that 
\begin{equation}
e^{-X_{i}^{n}}\phi_{ij}e^{X_{j}^{n}}=e^{B_{ij}^{n}}e^{X_{ij}^{n}}.\label{eq:2}
\end{equation}
For $n=1,$ this is the relation (\ref{eq:1}). Now, suppose this
is true for $n$. According to Proposition \ref{prop:Infinitesimal},
there exist $\left\{ \tilde{B}_{ij}^{n}\right\} _{ij}\in\mathcal{Z}^{1}\left(D,\mathcal{B}_{1}\left(\mathcal{F}\left[S_{d}\right]\right)^{\circ}\right)$
and $\left\{ Y_{i}^{n}\right\} _{i}\in\mathcal{Z}^{0}\left(D,\mathfrak{M}\cdot TS_{d}^{\circ}\right)$
such that 
\[
X_{ij}^{n}=Y_{i}^{n}+\tilde{B}_{ij}^{n}-Y_{j}^{n}.
\]

Taking the flow at time $1$ yields 
\begin{eqnarray*}
e^{-Y_{i}^{n}}e^{-X_{i}^{n}}\phi_{ij}e^{X_{j}^{n}}e^{Y_{j}^{n}} & = & e^{-Y_{i}^{n}}e^{B_{ij}^{n}}e^{X_{ij}^{n}}e^{Y_{j}^{n}}\\
 & = & e^{B_{ij}^{n}}\left[e^{-B_{ij}^{n}},e^{-Y_{i}^{n}}\right]e^{-Y_{i}^{n}}e^{X_{ij}^{n}}e^{Y_{j}^{n}}\\
 & = & e^{B_{ij}^{n}}\left[e^{-B_{ij}^{n}},e^{-Y_{i}^{n}}\right]e^{\tilde{B_{ij}^{n}}}e^{Y_{ij}^{n+1}}\\
 & = & e^{B_{ij}^{n}}e^{\tilde{B_{ij}^{n}}}\underbrace{e^{-\tilde{B_{ij}^{n}}}\left[e^{-B_{ij}^{n}},e^{-Y_{i}^{n}}\right]e^{\tilde{B_{ij}^{n}}}e^{Y_{ij}^{n+1}}}_{\in\mathfrak{G}\left(\mathfrak{M}^{n+1}\cdot TS_{d}^{\circ}\right)}\\
 & = & e^{B_{ij}^{n+1}}e^{X_{ij}^{n+1}},
\end{eqnarray*}
where $B_{ij}^{n+1}$ is given by the Campbell-Hausdorff (\ref{eq:camp})
where $X=B_{ij}^{n}$ and $Y=\tilde{B_{ij}^{n}}$, which ensures the
property by induction. Taking $n$ as big as necessary, the proposition
is a consequence of the stability property proved in \cite{MR2422017}.
\end{proof}
We can improve a bit the previous property taking advantage of the
inductive structure of the desingularization of $S_{d}.$ 
\begin{prop}
\label{prop:Induction}Let $E:\left(\mathcal{M},D\right)\to\left(\mathbb{C}^{2},0\right)$
be the desingularization of $\mathcal{F}\left[S_{d}\right]$. Consider
the sheaf $\mathfrak{I}\cdot TS_{d}$, where $\mathfrak{I}$ is the
ideal of functions vanishing along $D$ and $\mathcal{B}_{0}\left(\mathcal{F}\left[S_{d}\right]\right)=\ker\left(\left.\mathcal{B}\right|_{\mathfrak{I}\cdot TS_{d}}\right)$.
Then for every $\left\{ \phi_{ij}\right\} _{ij}\in Z^{1}\left(D,\mathfrak{G}\left(\mathfrak{I}\cdot TS_{d}^{\circ}\right)\right)$
there exists a family $\left\{ \psi_{ij}^{k}\right\} _{ij}$ $k=0\ldots l$
of $1$-cocycles in $Z^{1}\left(D,\mathfrak{G}\left(\mathcal{B}_{0}\left(\mathcal{F}\left[S_{d}\right]\right)^{\circ}\right)\right)$
such that 
\begin{equation}
\mathcal{M}\left[\phi_{ij}\right]\simeq\mathcal{M}\left[\psi_{ij}^{0}\right]\cdots\left[\psi_{ij}^{l}\right].\label{eq:basic-surg}
\end{equation}
In particular, $\mathcal{M}\left[\phi_{ij}\right]$ is the support
of a foliation obtained by successive basic surgeries of $\mathcal{F}\left[S_{d}\right]$.
\end{prop}

\begin{proof}
The proof is an induction on the length of the resolution of $S_{d}$.
Let us consider a $1$-cocyle $\left\{ \phi_{ij}\right\} _{ij}$ in
$\mathcal{Z}^{1}\left(\mathfrak{G}\left(\mathfrak{I}\cdot TS_{d}\right)^{\circ}\right).$
Let us consider $\left\{ \overline{\phi_{ij}}\right\} _{ij}$ the
restriction of the cocyle $\left\{ \phi_{ij}\right\} _{ij}$ to $D^{2}.$
We are going to apply inductively the property to $S_{2,d_{2}}$ for
some adapted direction $d_{2}$ of $S_{2}$ as defined in the proof
of Proposition \ref{proposition-fondamentale}. Applying inductively
Proposition \ref{prop:Induction} to $\left\{ \overline{\phi_{ij}}\right\} _{ij}$
yields the existence of $0-$cocycles in $\mathfrak{G}\left(\mathfrak{I}\cdot\left(TS_{d_{2}}^{2}\right)^{\circ}\right)$
and of $\mbox{1}-$cocycles $\mathfrak{G}\left(\mathcal{B}_{0}\left(\mathcal{F}\left[S_{d_{2}}^{2}\right]\right)^{\circ}\right)$
such that 
\[
\overline{\phi_{ij}}=\phi_{i}^{1}\psi_{ij}^{1}\phi_{i}^{2}\psi_{ij}^{2}\cdots\psi_{ij}^{M}\left(\phi_{j}^{M}\right)^{-1}\left(\phi_{j}^{M-1}\right)^{-1}\cdots\left(\phi_{j}^{1}\right)^{-1},
\]
a relation that is equivalent to (\ref{eq:basic-surg}) for $\left\{ \overline{\phi_{ij}}\right\} _{ij}$.
Now, consider the following $1-$cocyle 
\[
\tilde{\phi_{ij}}=\begin{cases}
\phi_{12}\phi_{j}^{1}\phi_{j}^{2}\cdots\phi_{j}^{M} & \textup{ for }i=1\textup{ and }j=2\\
\textup{Id} & \textup{ else}.
\end{cases}
\]
It belongs to $\mathcal{Z}^{1}\left(\mathfrak{G}\left(\mathfrak{I}\cdot TS_{d}\right)\right)$.
Since $\mathfrak{M}$ and $\mathfrak{I}$ coincide along $D_{1}$,
it belongs also to $\mathcal{Z}^{1}\left(\mathfrak{G}\left(\mathfrak{M}\cdot TS_{d}\right)^{\circ}\right)$.
Therefore, Proposition (\ref{prop:Induced-in-non-abelian}) yields
a $0-$cocycle and $1$-cocycle respectively in $\mathfrak{G}\left(\mathfrak{M}\cdot TS_{d}^{\circ}\right)$
and $\mathfrak{G}\left(\mathcal{B}_{1}\left(\mathcal{F}\left[S_{d}\right]\right)^{\circ}\right)$
such that 
\[
\tilde{\phi_{ij}}=\phi_{i}\psi_{ij}\phi_{j}^{-1}.
\]
In particular, if $\left(i,j\right)\neq2,$ then $\phi_{i}^{-1}\phi_{j}=\psi_{ij}$.
Therefore, for any $\left(i,j\right)\neq\left(1,2\right)$, one can
write 
\[
\phi_{ij}=\phi_{i}^{1}\psi_{ij}^{1}\phi_{i}^{2}\psi_{ij}^{2}\cdots\psi_{ij}^{M}\phi_{i}\psi_{ij}\phi_{j}^{-1}\left(\phi_{j}^{M}\right)^{-1}\left(\phi_{j}^{M-1}\right)^{-1}\cdots\left(\phi_{j}^{1}\right)^{-1}
\]
and 
\[
\phi_{12}=\phi_{1}\psi_{12}\phi_{2}^{-1}\left(\phi_{2}^{M}\right)^{-1}\left(\phi_{2}^{M-1}\right)^{-1}\cdots\left(\phi_{2}^{1}\right)^{-1}
\]
which is equivalent to (\ref{eq:basic-surg}) for $\left\{ \phi_{ij}\right\} _{ij}$.
The proposition is proved. 
\end{proof}
Finally, we can prove Theorem \ref{thm:Any-curve-C}. Let $E^{\prime}:\left(\mathcal{M^{\prime}},D^{\prime}\right)\to\left(\mathbb{C}^{2},0\right)$
be the desingularization of $\mathfrak{S}$. The curves $\mathfrak{S}$
and $S_{d}$ are topologically equivalent. Since $S$ is irreducible,
the exceptional divisors $D$ and $D^{'}$ are analytically equivalent.
Following \cite{MR2422017} section 3.2, there exists a $1$-cocycle
$\left\{ \phi_{ij}\right\} _{ij}$ in $\mathfrak{G}\left(\mathfrak{I}\cdot TS_{d}^{\circ}\right)$
such that 
\[
\mathcal{M}^{\prime}\simeq\mathcal{M}\left[\phi_{ij}\right].
\]
According to Proposition (\ref{prop:Induction}), $\mathcal{M}^{\prime}$
is the support of a foliation obtained from a basic surgery of $\mathcal{F}\left[S_{d}\right]$
that lets invariant the curve $C,$ which completes the proof of Proposition
\ref{thm:Any-curve-C}.

As a corollary, we obtain Theorem \ref{thm:If--is}, since under the
hypothesis mentionned, $p_{1}$ cannot be equal to $-1$ and $\mathcal{F}\left[S_{d}\right]$
is not dicritical along the exceptional divisor of the first blowing-up. 

\section{Theorem \ref{thm:If--is} $\protect\Longrightarrow$ Theorem \ref{conj2}.}

The proof consists in an argument by contradiction and four consecutive
steps.

\subsection{Step 1: construction of an equisingular family of curves $S\left(\epsilon\right)$
followed by an analytical family of forms $\omega\left(\epsilon\right)\in\Omega^{1}\left(S\left(\epsilon\right)\right)$
reaching the minimal valuation in $\Omega^{1}\left(S\left(\epsilon\right)\right)$}

$ $

Let $S$ be an irreducible germ of curve in the generic component
of its moduli space and let $E:\left(\mathcal{M},D\right)\to\left(\mathbb{C}^{2},0\right)$
be its minimal desingularization. Let $d$ be any direction for $S.$
Suppose that for a generic curve $\mathbb{S}_{d}$ in the topological
class of $S_{d}$, there exists a germ of $1$-form in $\Omega^{1}\left(\mathbb{S}_{d}\right)$
of multiplicity $\nu<\left[\frac{\nu\left(S_{d}\right)}{2}\right]$.
We suppose $\nu$ as small as possible with that property. We can
choose $\mathbb{S}_{d}$ in the generic stratum of the moduli space
$\mathbb{M}\left(S_{d}\right)$ so that, there exists an open neighborhood
$\mathbb{U}$ of $\mathbb{S}_{d}$ in $\mathbb{M}\left(S_{d}\right)$,
such that for any $C\in\mathbb{U}$, there exists a germ of $1$-form
in $\Omega^{1}\left(C\right)$ of multiplicity $\nu.$ Taking a local
parametrization of $\mathbb{M}\left(S_{d}\right)$ around $\mathbb{S}_{d,}$
\[
\epsilon\in\left(\mathbb{C}^{P},0\right)\to S_{d}\left(\epsilon\right)\in\mathbb{M}\left(S_{d}\right)
\]
with $S_{d}\left(0\right)=\mathbb{S}_{d}$, we obtain a universal
equisingular deformation of $\mathbb{S}_{d}.$ Moreover, for any $\epsilon\in\left(\mathbb{C}^{P},0\right),$
there exists $\omega\left(\epsilon\right)\in\Omega^{1}\left(S_{d}\left(\epsilon\right)\right)$
such that $\nu\left(\omega\left(\epsilon\right)\right)=\nu.$
\begin{lem}
We can suppose the family $\omega\left(\epsilon\right):\epsilon\in\left(\mathbb{C}^{P},0\right)\to\Omega^{1}\left(S_{d}\left(\epsilon\right)\right)$
being analytic in $\epsilon.$ 
\end{lem}

\begin{proof}
Up to some change of coordinates $\left(x,y\right)\in\left(\mathbb{C}^{2},0\right)$,
we can suppose that the direction $d$ is a fixed curve equal to $\emptyset,$
$\left\{ x=0\right\} $ or $\left\{ xy=0\right\} $ that does not
depend on $\epsilon$. In these three respective cases, any element
in $\Omega^{1}\left(S_{d}\left(\epsilon\right)\right)$ can be written
in coordinates 
\[
\omega\left(\epsilon\right)=\left\{ \begin{array}{ll}
A_{\epsilon}\dd x+B_{\epsilon}\dd y, & \ d=\emptyset\\
A_{\epsilon}\dd x+xB_{\epsilon}\dd y, & \ d=\left\{ x=0\right\} \\
\textup{or }\\
yA_{\epsilon}\dd x+xB_{\epsilon}\dd y, & \ d=\left\{ xy=0\right\} 
\end{array}\right.
\]

Let $\gamma_{\epsilon}$ be a Puiseux parametrization of $S\left(\epsilon\right)$
depending analytically on $\epsilon.$ The hypothesis ensures that
for any $M\in\mathbb{N}$ and for any $\epsilon,$ the following system
has a solution $\omega$
\[
\left(S_{\epsilon}\right):\left\{ \begin{array}{lc}
\textup{Jet}_{t=0}^{M}\left(\gamma_{\epsilon}^{*}\omega\right)=0 & \left(1\right)\\
\textup{Jet}_{\left(x,y\right)}^{\nu-1}\omega=0 & \left(2\right)\\
\textup{Jet}_{\left(x,y\right)}^{\nu}\omega\neq0 & \left(3\right)
\end{array}\right..
\]

The family $\left(S_{\epsilon}\right)_{\epsilon}$ is an analytical
family of linear systems with a finite number of unknown variables,
say $M$, which are some coefficients of the Taylor expansion of $A_{\epsilon}$
and $B_{\epsilon}$ - $\left(1\right)$ and $\left(2\right)$ - and
an open condition $\left(3\right)$. The solutions can be viewed as
a semi-analytic set $Z$ of $\mathbb{C}^{M+P}$ that projects onto
$L\subset\mathbb{C}^{P}$ through the projection $\mathbb{C}^{M+P}\to\mathbb{C}^{P}.$
Hence, there exists a germ of analytical section $\sigma:\left(L,l\right)\to\mathbb{C}^{M+P}$
defined near some $l\in L$ such that for all $\epsilon\in\left(L,l\right),$
one has $\sigma\left(\epsilon\right)\in Z$. This provides two functions
$A_{\epsilon}$ and $B_{\epsilon}$ in $\mathbb{C}\left\{ \epsilon\right\} $$\left[x,y\right]$
such that $\omega_{\epsilon}$ is a solution of $\left(S_{\epsilon}\right).$
Since the family $\gamma_{\epsilon}$ is topologically trivial, taking
a bigger integer $M$ if necessary, we can find a family of functions
$f_{k}\in\mathbb{C}\left\{ x,y\right\} $ with $\nu\left(df_{k}\right)>\nu,$
$\nu\left(df_{k}\right)\xrightarrow[k\to\infty]{}+\infty$ such that
for any $k\geq M$ and any $\epsilon$, one has 
\[
\nu\left(\gamma_{\epsilon}^{*}df_{k}\right)=k.
\]
Considering a form written 
\begin{equation}
\Omega=\omega_{\epsilon}+\sum_{k\geq M}\alpha_{k}\left(\epsilon\right)\dd f_{k},\label{eq:25}
\end{equation}
we can choose inductively $\alpha_{k}\left(\epsilon\right)$ such
that (\ref{eq:25}) becomes a formal solution $\Omega\in\mathbb{C}\left\{ \epsilon\right\} \left[\left[x,y\right]\right]$
of the system 
\[
\left\{ \begin{array}{l}
\gamma_{\epsilon}^{*}\Omega=0\\
\textup{Jet}_{\left(x,y\right)}^{\nu-1}\Omega=0\\
\textup{ and}\\
\textup{Jet}_{\left(x,y\right)}^{\nu}\Omega\neq0
\end{array}\right..
\]
According to the Artin's approximation theorem \cite{Artin}, we can
take $\Omega$ analytic as a whole, $\Omega\in\mathbb{C}\left\{ \epsilon,x,y\right\} $. 
\end{proof}
For $\epsilon$ generic, we can also suppose that $\omega\left(\epsilon\right)$
is equireducible \cite{MatSal}. Let 
\[
E\left(\epsilon\right):\left(\mathcal{M}\left(\epsilon\right),D\left(\epsilon\right)\right)\to\left(\mathbb{C}^{2},0\right)
\]
be the equisingular family of minimal desingularizations of the foliations
$\mathcal{F}\left(\epsilon\right)$ defined by $\omega\left(\epsilon\right).$
In particular, $E\left(\epsilon\right)$ is also an equisingular family
of desingularizations of $S_{d}\left(\epsilon\right).$ For the sake
of simplicity, we still denote by $\mathcal{M}$, $E$ and $S_{d}$
respectively the manifold $\mathcal{M}\left(0\right),$ the desingularization
$E\left(0\right)$ and the curve $S_{d}\left(0\right)$.

\subsection{Step 2: vanishing of some cohomology. }

$ $

Let $\left\{ T_{ij}\right\} _{ij}$ be a $1-$cocycle in $\mathcal{Z}^{1}\left(\mathcal{M},TS_{d}\right).$
Let us consider the deformation obtained by the gluing 
\[
\mathcal{M}\left[e^{\left(t\right)T_{ij}}\right].
\]
Since the flow $e^{\left(t\right)T_{ij}}$ lets globally invariant
$S_{d},$ the manifold $\mathcal{M}\left[e^{\left(t\right)T_{ij}}\right]$
admits an invariant curve topologically equivalent to $S_{d}^{E}$.
By versality, the so defined topologically trivial deformation is
equivalent to a deformation $S_{d}\left(\epsilon\left(t\right)\right)$
for some analytic factorization $\epsilon\left(t\right):\left(\mathbb{C},0\right)\to\left(\mathbb{C}^{P},0\right)$.
The deformation $S_{d}\left(\epsilon\left(t\right)\right)$ is followed
by the deformation of foliations $\mathcal{F}\left(\epsilon\left(t\right)\right)$.
Therefore on the open set $\mathcal{M}\left(\epsilon\right)^{*}$
which is $\mathcal{M\left(\epsilon\right)}$ deprived of the singular
locus of $E\left(\epsilon\right)^{*}\mathcal{F}\left(\epsilon\right)$
, the cocycle $\left\{ e^{\left(t\right)T_{ij}}\right\} _{ij}$ is
equivalent to a cocycle of basic automorphisms. Thus, there exist
a $0-$cocycle of automorphism $\left\{ \phi_{i}\left(t\right)\right\} _{i}$
letting globally invariant $S\left(\epsilon\left(t\right)\right)_{d}^{E\left(\epsilon\left(t\right)\right)}$
and $D\left(\epsilon\left(t\right)\right)$ and a $1-$cocycle of
basic automorphisms $\left\{ B_{ij}\left(t\right)\right\} _{ij}$
for $\mathcal{F}$, such that on $\mathcal{M}\left(\epsilon\left(t\right)\right)^{*}$,
one has 
\[
e^{\left(t\right)T_{ij}}=\phi_{i}\left(t\right)B_{ij}\left(t\right)\phi_{j}^{-1}\left(t\right).
\]
Taking the derivative at $t=0$ of the above expression yields to
a cohomogical relation on $\mathcal{M}\left(0\right)=\mathcal{M}.$
\begin{equation}
T_{ij}=T_{i}+b_{ij}-T_{j}\label{eq:onto}
\end{equation}
where $\left\{ T_{i}\right\} $ is a $0-$cocycle in $TS_{d}$ and
$\left\{ b_{ij}\right\} _{ij}$ is a $1-$cocycle with values in the
sub-sheaf of basic vector fields for $\mathcal{F}$ tangent to $S_{d},$
denoted simply by $\mathcal{B}\left(\mathcal{F}\right)$. 

Let us denote by $\Omega$ the image sheaf of $TS_{d}$ by the basic
operator (\ref{eq:142}) for $\mathcal{F}$ with a given balanced
equation $F$. 

The following diagram\begin{equation}\label{diagram}\xymatrix{&H^1\left(\mathcal{M}^*,\mathcal{B}\left(\mathcal{F}\right)\right)\ar[d]^i\\H^1\left(\mathcal{M},TS_d\right)\ar[r]^\alpha\ar[d]^\beta &H^1\left(\mathcal{M}^*,TS_d\right)\ar[d]^{\mathcal{B}}\\ H^1\left(\mathcal{M},\Omega\right)\ar[d]^\delta\ar[r]^\gamma &H^1\left(\mathcal{M}^*,\Omega\right)\\H^2\left(\mathcal{M},\mathcal{B}\left(\mathcal{F}\right)\right) &} \end{equation}
is commutative. Since for any $1$-cocycle $\left\{ T_{ij}\right\} _{ij}\in Z^{1}\left(\mathcal{M},TS_{d}\right)$,
a relation such as \ref{eq:onto} exists, one has 
\[
\textup{Im}\alpha\subset\textup{Im}i.
\]
Thus, the composed map $\mathcal{B}\circ\alpha$ is the zero map.
\\
The sheaf $\Omega$ on $\mathcal{M}^{*}$ can be described as follows
\[
\Omega=\Omega^{2}\left(2\left(F\right)^{E}-S_{d}^{E}+\sum n_{i}D_{i}\right)
\]
where $D=\sum D_{i}$ and the $n_{i}$'s are some integers depending
on $\mathcal{F}.$ This sheaf can be extended analytically on $\mathcal{M}.$
The Mayer-Vietoris sequence applied to the covering $\left\{ \mathcal{M}^{*},\mathcal{U}\right\} $
of $\mathcal{M}$ where $\mathcal{U}$ is an union of some small open
balls around each singularity is written 
\begin{multline*}
\cdots\to H^{0}\left(\mathcal{M}^{*},\Omega\right)\bigoplus H^{0}\left(\mathcal{U},\Omega\right)\xrightarrow{\Delta}H^{0}\left(\mathcal{\mathcal{M}}^{*}\cap\mathcal{U},\Omega\right)\\
\to H^{1}\left(\mathcal{\mathcal{M}},\Omega\right)\to H^{1}\left(\mathcal{\mathcal{M}}^{*},\Omega\right)\bigoplus H^{1}\left(\mathcal{U},\Omega\right)\to\cdots
\end{multline*}
The Hartogs's extension result ensures that $\Delta$ is onto. Moreover,
since $\mathcal{U}$ can be supposed to be Stein and $\Omega$ is
coherent, we deduce that in the diagram (\ref{diagram}) the map $\gamma$
is injective. 
\begin{prop}
\label{H2}We have 
\[
H^{2}\left(\mathcal{M},\mathcal{B}\left(\mathcal{F}\right)\right)=0.
\]
\end{prop}

\begin{proof}
Taking small flow-boxes on the regular part of $E^{*}\mathcal{F}$,
we can find a finite Stein covering $\left\{ U_{\alpha}\right\} _{\alpha\in I}\cup\left\{ U_{s}\right\} _{s\in\textup{Sing}\left(E^{*}\mathcal{F}\right)}$
of $\mathcal{M}$ such that 

\begin{itemize}
\item for any $s\in\textup{Sing}\left(E^{*}\mathcal{F}\right)$, $U_{s}$
is a \emph{very} small neighborhood of $s$.
\item on any open set $U_{\alpha}$ with $\alpha\in I$, there exists a
biholomorphism on its image $\psi_{\alpha}:U_{\alpha}\to\mathbb{C}^{2}$
such that $\left(\psi_{\alpha}^{-1}\right)^{*}\left.E^{*}\mathcal{F}\right|_{U_{\alpha}}$
is the trivial regular foliation given by $\dd x=0.$
\end{itemize}
Applying the Mayer-Vietoris sequence to the covering leads to the
long exact sequence in cohomology from which is extracted 
\begin{multline}
\underset{\alpha,\beta}{\bigoplus}H^{1}\left(U_{\alpha}\cap U_{\beta},\mathcal{B}\left(\mathcal{F}\right)\right)\to H^{2}\left(\mathcal{M},\mathcal{B}\left(\mathcal{F}\right)\right)\to\\
\underset{\alpha\in I}{\bigoplus}H^{2}\left(U_{\alpha},\mathcal{B}\left(\mathcal{F}\right)\right)\underset{s\in\textup{Sing}\left(E^{*}\mathcal{F}\right)}{\bigoplus}H^{2}\left(U_{s},\mathcal{B}\left(\mathcal{F}\right)\right)\to\underset{\alpha,\beta}{\bigoplus}H^{2}\left(U_{\alpha}\cap U_{\beta},\mathcal{B}\left(\mathcal{F}\right)\right)\label{eq:exactseqencore}
\end{multline}
Now, one has the following lemma.
\begin{lem*}
Let $U$ a Stein open set in $\mathbb{C}^{2}$ foliatied by the foliation
$\mathcal{F}$ given by $\dd x.$ Then for any $i\geq1,$
\[
H^{i}\left(U,\mathcal{B}\left(\mathcal{F}\right)\right)=0.
\]
\end{lem*}
\begin{proof}
Once the coordinates $\left(x,y\right)$ are given, a basic vector
field $X$ for $\mathcal{F}$ satisfies 
\[
L_{X}\dd x\wedge\dd x=0.
\]
It is uniquely written $X=a\left(x\right)\frac{\partial}{\partial x}+b\left(x,y\right)\frac{\partial}{\partial y}$.
Thus, the sheaf $\mathcal{B}\left(\mathcal{F}\right)$ is isomorphic
to the direct sum of sheaves $\mathcal{O}_{1}\frac{\partial}{\partial x}\oplus\mathcal{O}_{2}\frac{\partial}{\partial y}$
whose cohomology in rank greater than $1$ is trivial on Stein open
sets. 
\end{proof}
The foliation is analytically equivalent to the trivial foliation
$\dd x=0$ on any 2-intersection $U_{\alpha}\cap U_{\beta}$ and on
$U_{\alpha}$ with $\alpha\in I$. Therefore, the cohomology of $\mathcal{B}\left(\mathcal{F}\right)$
vanishes in rank $1$ and $2$ on these open sets. Finally, (\ref{eq:exactseqencore})
is written 
\[
H^{2}\left(\mathcal{M},\mathcal{B}\left(\mathcal{F}\right)\right)\simeq\underset{s\in\textup{Sing}\left(E^{*}\mathcal{F}\right)}{\oplus}H^{2}\left(U_{s},\mathcal{B}\left(\mathcal{F}\right)\right).
\]
The open sets $U_{s}$ can be taken as small as needed. Thus, the
inductive limit \cite{Godement} on the family of open sets containing
the singular locus of $E^{*}\mathcal{F}$ is written 
\[
0=\underset{s\in\textup{Sing}\left(E^{*}\mathcal{F}\right)}{\bigoplus}H^{2}\left(\left\{ s\right\} ,\mathcal{B}\left(\mathcal{F}\right)\right)\simeq\underset{s\in\textup{Sing}\left(E^{*}\mathcal{F}\right)}{\bigoplus}\lim_{U_{s}\to s}H^{2}\left(U_{s},\mathcal{B}\left(\mathcal{F}\right)\right)\simeq H^{2}\left(\mathcal{M},\mathcal{B}\left(\mathcal{F}\right)\right).
\]
from which the lemma follows. 
\end{proof}
The previous lemma and the properties of the diagram (\ref{diagram})
ensure that
\[
H^{1}\left(\mathcal{M},\Omega\right)=0.
\]
Now, let us consider $E_{1}:\left(\mathcal{M}_{1},D_{1}\right)\to\left(\mathbb{C}^{2},0\right)$
the first blowing-up in the resolution $E.$ The Mayer-Vietoris sequence
of an adapted covering shows that
\begin{multline*}
H^{1}\left(\mathcal{M}_{1},\Omega^{2}\left(2\left(F\right)^{E_{1}}-S_{d}^{E_{1}}+n_{1}D_{1}\right)\right)\\
\hookrightarrow H^{1}\left(\mathcal{M},\Omega^{2}\left(2\left(F\right)^{E}-S_{d}^{E}+\sum n_{i}D_{i}\right)\right)
\end{multline*}
and therefore, 
\begin{equation}
H^{1}\left(\mathcal{M}_{1},\Omega^{2}\left(2\left(F\right)^{E_{1}}-S_{d}^{E_{1}}+n_{1}D_{1}\right)\right)=0.\label{eq:contra}
\end{equation}

\subsection{Step 3: the contradiction.}

$ $

We are going to prove that the vanishing (\ref{eq:contra}) leads
to a contradiction with $\nu\left(\mathcal{F}\right)<\left[\frac{\nu\left(S_{d}\right)}{2}\right].$ 

We recall that $F$ being a balanced equation of $\mathcal{F}$ \cite{Genzmer1},
the next relation holds 
\[
\nu\left(\mathcal{F}\right)=\nu\left(F\right)-1+\tau\left(\mathcal{F}\right)
\]
where $\tau\left(\mathcal{F}\right)$ is a positive integer called
\emph{the tangency excess of $\mathcal{F}$. }
\begin{itemize}
\item Suppose that $\mathcal{F}$ is not dicritical along the exceptional
divisor of the blowing-up of its singularity. A computation in coordinates
ensures that $n_{1}=1-2\tau\left(\mathcal{F}\right).$ However, if
(\ref{eq:contra}) is true, Lemma \ref{lemma1} shows that 
\[
2\nu\left(F\right)-\nu\left(S_{d}\right)\geq n_{1}\Longleftrightarrow2\nu\left(\mathcal{F}\right)-\nu\left(S_{d}\right)\geq-1.
\]
But $\nu\left(\mathcal{F}\right)\leq\left[\frac{\nu\left(S_{d}\right)}{2}\right]-1$
gives us 
\[
2\nu\left(\mathcal{F}\right)-\nu\left(S_{d}\right)\leq2\left[\frac{\nu\left(S_{d}\right)}{2}\right]-\nu\left(S_{d}\right)-2<-1
\]
which is a contradiction.
\item Suppose now that $\mathcal{F}$ is dicritical along the exceptional
divisor of the single blowing-up of its singularity. Then $n_{1}=-2\tau\left(\mathcal{F}\right).$
Again, Lemma \ref{lemma1} ensures that 
\[
2\nu\left(\mathcal{F}\right)-\nu\left(S_{d}\right)\geq-2.
\]
If $\nu\left(\mathcal{F}\right)\leq\left[\frac{\nu\left(S_{d}\right)}{2}\right]-2$
then we are led to a contradiction. Suppose that $\nu\left(\mathcal{F}\right)=\left[\frac{\nu\left(S_{d}\right)}{2}\right]-1.$
If $\nu\left(S_{d}\right)$ is odd then 
\[
2\nu\left(\mathcal{F}\right)-\nu\left(S_{d}\right)=2\left(\frac{\nu\left(S_{d}\right)-1}{2}-1\right)-\nu\left(S_{d}\right)=-3,
\]
which is still a contradiction. Suppose that $\nu\left(S_{d}\right)$
is even. Then, $\nu\left(\mathcal{F}\right)=\frac{\nu\left(S_{d}\right)}{2}-1.$
The multiplicity $\nu\left(\mathcal{F}\right)$ being as small as
possible in $\Omega^{1}\left(S_{d}\right)$, a basis of $\Omega^{1}\left(S_{d}\right)$
can be written $\left\{ \omega_{1},\omega_{2}\right\} $ with 
\[
\frac{\nu\left(S_{d}\right)}{2}-1=\nu\left(\omega_{1}\right)\leq\nu\left(\omega_{2}\right)\textup{ and }\nu\left(\omega_{1}\right)+\nu\left(\omega_{2}\right)\leq\nu\left(S_{d}\right).
\]
 Thus there are only three possibilities for $\nu\left(\omega_{2}\right).$

\begin{itemize}
\item if $\nu\left(\omega_{2}\right)=\frac{\nu\left(S_{d}\right)}{2}+1$,
then any $1$-form $\omega$ of multiplicity $\frac{\nu\left(S_{d}\right)}{2}$
in $\Omega^{1}\left(S_{d}\right)$ is written 
\[
\omega=a\omega_{1}+b\omega_{2},
\]
where $a$ is a function of multiplicity $1$ and $b$ is any function.
In particular, its jet of smallest order is written 
\[
\left(a\right)_{1}\cdot\left(\omega_{1}\right)_{\frac{\nu\left(S_{d}\right)}{2}},
\]
where $\left(\star\right)_{i}$ stands for the jet of order $i.$
Thus, as $\omega_{1}$ is dicritical along the exceptional divisor
of the single blowing-up of its singularity, $\omega$ is also. This
would imply that any element of multiplicity $\frac{\nu\left(S_{d}\right)}{2}$
in the Saito module has this property. This is a contradiction with
Theorem \ref{thm:If--is}.
\item if $\nu\left(\omega_{2}\right)=\frac{\nu\left(S_{d}\right)}{2}$ or
$\nu\left(\omega_{2}\right)=\frac{\nu\left(S_{d}\right)}{2}-1$ then
using the criterion of Saito we have 
\[
\left(\omega_{1}\right)_{\nu\left(\omega_{1}\right)}\wedge\left(\omega_{2}\right)_{\nu\left(\omega_{2}\right)}=0.
\]
Therefore, $\omega_{2}$ is dicritical after one blowing-up. If $\nu\left(\omega_{2}\right)=\frac{\nu\left(S_{d}\right)}{2}$
then any $1-$form of multiplicity $\frac{\nu\left(S_{d}\right)}{2}$
is dicritical, which is impossible. If $\nu\left(\omega_{2}\right)=\frac{\nu\left(S_{d}\right)}{2}-1$,
let us write 
\begin{eqnarray*}
\omega_{1} & = & P_{1}\omega_{r}+\cdots\\
\omega_{2} & = & P_{2}\omega_{r}+\cdots
\end{eqnarray*}
 where $\omega_{r}=x\dd y-y\dd x$. Consider $\omega$ in the module
of Saito with multiplicity $\frac{\nu\left(S_{d}\right)}{2}.$ It
can be written 
\[
\omega=a\omega_{1}+b\omega_{2}=\left(aP_{1}+bP_{2}\right)\omega_{r}+\cdots.
\]
If $\nu\left(a\right)=0$ or $\nu\left(b\right)=0$ then $\nu\left(\omega\right)=\frac{\nu\left(S_{d}\right)}{2}-1$
unless there exists a non vanishing constant $C$ such that $P_{1}=CP_{0}.$
But in that latter case $\left\{ \omega_{0},\omega_{1}-C\omega_{0}\right\} $
is still a basis of the module of Saito with $\nu\left(\omega_{1}-C\omega_{0}\right)>\frac{\nu\left(S_{d}\right)}{2}-1$
which leads to a case already treated. Thus, $\nu\left(a\right)\geq1$
and $\nu\left(b\right)\geq1$ and necessarily, $\omega$ is dicritical
along the exceptional divisor of one blowing-up. As before, any $1$-form
of multiplicity $\frac{\nu\left(S_{d}\right)}{2}$ would be dicritical
along the exceptional divisor of the blowing-up of its singularity,
which is impossible. 
\end{itemize}
\end{itemize}
This completes the proof of the first part of Theorem \ref{conj2},
and thus for $S$ generic, we prove that
\begin{equation}
\min_{\omega\in\Omega^{1}\left(S_{d}\right)}\nu\left(\omega\right)=\left[\frac{\nu\left(S_{d}\right)}{2}\right].\label{eq:hyperimp}
\end{equation}

\subsection{Step 4: existence of a balanced basis. }

Let us prove now the existence of balanced basis for $\Omega^{1}\left(S_{d}\right).$

Let us suppose first that $\nu\left(S_{d}\right)$ is even. Consider
a basis $\left\{ \omega_{1},\omega_{2}\right\} $ of $\Omega^{1}\left(S_{d}\right)$.
According to (\ref{eq:hyperimp}) there are some $1$-forms with multiplicity
$\frac{\nu\left(S_{d}\right)}{2}$ in $\Omega^{1}\left(S_{d}\right)$.
Hence, at least one of the forms in the basis, say $\omega_{1}$,
has a multiplicity equal to $\frac{\nu\left(S_{d}\right)}{2}$. The
multiplicity of $\omega_{2}$ is greater or equal to $\frac{\nu\left(S_{d}\right)}{2}.$
If it is equal, then the basis is balanced. If not, $\left\{ \omega_{1},\omega_{1}+\omega_{2}\right\} $
is still a basis and is balanced. 

Suppose now that $\nu\left(S_{d}\right)$ is odd. If the direction
of $S_{d}$ is empty or contains one component, let us consider $\tilde{S}=S_{d}\cup L$
where $L$ is a smooth curve transverse to the direction of $S_{d}.$
Since the multiplicity of $\tilde{S}$ is even, according to the previous
case, the module $\Omega^{1}\left(\tilde{S}\right)$ admits a balanced
basis. Therefore there exists a couple a $1-$forms $\left\{ \omega_{1},\omega_{2}\right\} $
of multiplicity $\frac{\nu\left(S_{d}\right)+1}{2}$ such that 
\[
\omega_{1}\wedge\omega_{2}=ulf\dd x\wedge\dd y,\qquad u\left(0\right)\neq0.
\]
where $l$ is an irreducible equation of $L$ and $f$ a reduced equation
of $S_{d}$. Now, according to (\ref{eq:hyperimp}), there exists
$\overline{\omega}$ tangent to $S_{d}$ such that $\nu\left(\overline{\omega}\right)=\frac{\nu\left(S_{d}\right)-1}{2}.$
The $1$-form $l\overline{\omega}$ is tangent to $\tilde{S}$. Hence,
there exist two germs of functions $a_{1}$ and $a_{2}$ such that
\[
l\overline{\omega}=a_{1}\omega_{1}+a_{2}\omega_{2}.
\]
The functions $a_{1}$ and $a_{2}$ cannot both vanish. Suppose by
symmetry that $a_{1}$ does not vanish, then $\left\{ l\overline{\omega},\omega_{2}\right\} $
is a basis of $\Omega^{1}\left(\tilde{S}\right).$Thus 
\[
l\overline{\omega}\wedge\omega_{2}=vlf\dd x\wedge\dd y,\qquad v\left(0\right)\neq0.
\]
Dividing by $l$ the above expression leads to the criterion of Saito
for the balanced basis $\left\{ \overline{\omega},\omega_{2}\right\} $
of $\Omega^{1}\left(S_{d}\right)$. 

If the direction of $S_{d}$ contains two components $L_{1}$ and
$L_{2}$, then let us consider $\tilde{S}=S\cup L_{1}.$ The module
$\Omega^{1}\left(\tilde{S}\right)$ admits a balanced basis $\left\{ \omega_{1},\omega_{2}\right\} $
with $\nu\left(\omega_{1}\right)=\nu\left(\omega_{2}\right)=\left[\frac{\nu\left(\tilde{S}\right)}{2}\right]=\frac{\nu\left(S\right)+1}{2}.$
Now, there exist $\overline{\omega}$ in $\Omega^{1}\left(S_{d}\right)$
with $\nu\left(\overline{\omega}\right)=\left[\frac{\nu\left(S_{d}\right)}{2}\right]=\frac{\nu\left(S\right)+1}{2}.$
Since $\overline{\omega}$ is also tangent to $S\cup L_{1},$ there
exist two functions $a_{1}$ and $a_{2}$ such that 
\[
\overline{\omega}=a_{1}\omega_{1}+a_{2}\omega_{2}.
\]
The functions $a_{1}$ and $a_{2}$ cannot both vanish so we can suppose
that $a_{1}\left(0\right)\neq0$. The family $\left\{ \overline{\omega},\omega_{2}\right\} $
is still a basis of $\Omega^{1}\left(\tilde{S}\right)$ that satisfies
\[
\overline{\omega}\wedge\omega_{2}=wfl_{1}\dd x\wedge\dd y,\qquad w\left(0\right)\neq0.
\]
Thus, multiplying by $l_{2}$ leads to 
\[
\overline{\omega}\wedge l_{2}\omega_{2}=wfl_{1}l_{2}\dd x\wedge\dd y,\qquad w\left(0\right)\neq0
\]
and $\left\{ \overline{\omega},l_{2}\omega_{2}\right\} $ is a balanced
basis of $\Omega^{1}\left(S_{d}\right)$. 

This ends the proof of Theorem \ref{conj2}. 

\bibliographystyle{plain}
\bibliography{/home/genzmer/ownCloud/Article/Biblio/Bibliographie}

\address{Yohann Genzmer\\
Institut de Mathmatiques de Toulouse\\
118 Route de Narbonne\\
31062 Toulouse\\
France\\
\textsf{yohann.genzmer@math.univ-toulouse.fr}}
\end{document}